\documentclass[11pt]{amsart}

\usepackage[text={450pt,680pt},centering]{geometry}

\usepackage{color}
\usepackage{esint,amssymb}
\usepackage{graphicx}
\usepackage{MnSymbol}
\usepackage{mathtools}
\usepackage[colorlinks=true, pdfstartview=FitV, linkcolor=blue, citecolor=blue, urlcolor=blue,pagebackref=false]{hyperref}
\usepackage{microtype}

\usepackage{tikz}

\definecolor{darkgreen}{rgb}{0,0.5,0}
\definecolor{darkblue}{rgb}{0,0,0.7}
\definecolor{darkred}{rgb}{0.9,0.1,0.1}

\newtheorem{theorem}{Theorem}
\newtheorem{proposition}[theorem]{Proposition}
\newtheorem{lemma}[theorem]{Lemma}
\newtheorem{corollary}[theorem]{Corollary}

\theoremstyle{definition}
\newtheorem{remark}[theorem]{Remark}

\newtheorem{assumption}{Assumption}

\newcommand{\tref}[1]{Theorem~\ref{t.#1}}
\newcommand{\pref}[1]{Proposition~\ref{p.#1}}
\newcommand{\lref}[1]{Lemma~\ref{l.#1}}
\newcommand{\cref}[1]{Corollary~\ref{c.#1}}

\newcommand{\rref}[1]{Remark~\ref{r.#1}}
\newcommand{\eref}[1]{(\ref{e.#1})}

\numberwithin{equation}{section}
\numberwithin{theorem}{section}

\newcommand{\K}{\mathbb{K}_{u_0}}
\newcommand{\Z}{\mathbb{Z}}
\newcommand{\C}{\mathbb{C}}

\newcommand{\N}{\mathbb{N}}
\newcommand{\R}{\mathbb{R}}

\renewcommand{\P}{\mathbb{P}}
\newcommand{\F}{\mathcal{F}}

\newcommand{\LL}{\mathbb{L}}
\newcommand{\e}{\varepsilon}

\newcommand{\I}{\mathcal{I}}
\newcommand{\Hh}{\mathcal{H}}
\newcommand{\Hhr}{\mathfrak{H}}

\newcommand{\loc}{\mathrm{loc}}
\newcommand{\Id}{\mathrm{Id}}

\renewcommand{\subset}{\subseteq}

\newcommand{\expec}[1]{\mathbb{E}\left[{#1}\right]}
\newcommand{\Msym}{\R_{\mathrm{sym}}^{d\times d}}
\newcommand{\Md}{\R^{d\times d}}
\newcommand{\Mdd}{(\R^{d\times d})^2_{\mathrm{sym}}}

\renewcommand{\L}{\mathcal{L}}
\renewcommand{\fint}{\strokedint}
\newcommand{\per}{\mathrm{per}}

\DeclareMathOperator{\dist}{dist}

\DeclareMathOperator{\supp}{supp}

\renewcommand{\bar}{\overline}
\renewcommand{\tilde}{\widetilde}

\begin{document}

\title[Loss of strong ellipticity through homogenization]{Loss of strong ellipticity through homogenization in 2D linear elasticity: A phase diagram}

\begin{abstract}
Since the seminal contribution of Geymonat, M\"uller, and Triantafyllidis, it is known that strong ellipticity is not necessarily conserved through periodic homogenization in linear elasticity.
This phenomenon is related to microscopic buckling of composite materials. 
Consider a mixture of two isotropic phases which leads to loss of strong ellipticity when arranged in a laminate manner, as considered
by Guti\'errez and by Briane and Francfort.
In this contribution we prove that the laminate structure is essentially the only microstructure which leads to such a loss of strong ellipticity. We perform a more general analysis in the stationary, ergodic setting.
\end{abstract}

\author[A. Gloria]{Antoine Gloria}
\address[Antoine Gloria]{Sorbonne Universit\'e, UMR 7598, Laboratoire Jacques-Louis Lions, Paris, France \& Universit\'e Libre de Bruxelles, Brussels, Belgium}\email{gloria@ljll.math.upmc.fr}

\author[M. Ruf]{Matthias Ruf}
\address[Matthias Ruf]{Universit\'e Libre de Bruxelles (ULB), Brussels, Belgium}
\email{matthias.ruf@ulb.ac.be}

\subjclass[2010]{}
\date{\today}

\maketitle


\section{Introduction}

Consider a $Q=[0,1)^d$-periodic heterogeneous linear elastic material characterized by its elasticity tensor field $\LL:\R^d \to \Mdd$ (symmetric fourth-order tensors).
Assume that $\LL$ is pointwise very strongly elliptic, i.~e. $M \cdot \LL(x) M \ge \lambda |M|^2$ for some $\lambda>0$, all $M\in \Msym$, and almost all $x\in Q$.
Let $D$ be a Lipschitz bounded domain of $\R^d$, and $u_0 \in H^1(D)$.
Then the integral functional  $H^1_0(D)\ni u \mapsto \I_\e(u):=\int_D \nabla (u+u_0) \cdot \LL(\frac{\cdot}{\e})\nabla (u+u_0)$ $\Gamma$-converges for the weak topology of $H^1(D)$ 
to the homogenized integral functional $H^1_0(D)\ni u \mapsto \I_*(u):=\int_D \nabla (u+u_0) \cdot \LL_* \nabla (u+u_0)$, where $\LL_*$ is a constant elasticity tensor that is very strongly elliptic with constant $\lambda$, see for instance \cite{Tartar79,Francfort-Murat,Francfort92,JKO}.

\medskip

If instead of pointwise very strong ellipticity, we only assume pointwise strong ellipticity, i.~e. $M \cdot \LL(x) M \ge \lambda |M|^2$ for some $\lambda>0$, all \emph{rank-one} matrices $M=a\otimes b \in \Md$, and almost all $x\in Q$, the story is different. In an inspirational work \cite{GMT-93}, Geymonat, M\"uller, and Triantafyllidis indeed showed that three phenomena can occur:
\begin{enumerate}
\item $\I_\e$ is not uniformly bounded from below, and there is no homogenization;
\item $\I_\e$ is uniformly bounded from below, there is homogenization towards $\I_*$, and $\LL_*$ is strongly elliptic (that is, non-degenerate on rank-one matrices);
\item $\I_\e$ is uniformly bounded from below, there is homogenization towards $\I_*$, and $\LL_*$ is non-negative on rank-one matrices, but not strongly elliptic (there exists 
$a \otimes b\neq 0$ such that $a \otimes b\cdot \LL_* a\otimes b=0$).
\end{enumerate}
The third phenomenon is referred to as \emph{loss of strong ellipticity through homogenization}.
To avoid confusion we will say that a fourth-order tensor $\LL$ is strongly elliptic if $M\cdot \LL M\ge 0$ for all rank-one matrices, and that it is \emph{strictly} strongly elliptic if in addition there exists $\lambda>0$ such that this inequality can be strengthened to $M\cdot \LL M\ge \lambda |M|^2$ for rank-one matrices $M$.

\medskip

There is essentially one single example in the literature for which one can prove that strong ellipticity is lost through homogenization in dimension $d=2$, to which
we restrict in the third (and main) section of this article.
The associated composite material has a laminate structure made of two isotropic phases (a strong phase and a weak phase). 
Loss of strong ellipticity occurs when the strong phase buckles in compression (it is somehow related to the failure of the cell-formula for nonlinear composites, cf. \cite{Muller-87,GMT-93}),
and has been rigorously established in \cite{Gutierrez-98,Briane-Francfort-15}.

\medskip

Buckling is a one-dimensional phenomenon, and it seems unlikely from a mechanical point of view that a material could lose strong ellipticity in \emph{every} rank-one direction.
This elementary observation suggests that assuming the isotropy of $\LL_*$ may prevent loss of strong ellipticity through homogenization. In the two-dimensional periodic setting, this fact was proven by Francfort and the first author in \cite{Francfort-Gloria-16}. However, a closer look at the argument reveals that the loss of ellipticity is indeed prevented by the connectedness of the weak phase rather than by isotropy of $\LL_*$.
The aim of the present contribution is to investigate the interplay between the connectedness of phases and the loss of strong ellipticity through homogenization, in the more general setting of stochastic homogenization. 

\medskip

Denote by $\LL^1$ and $\LL^2$ the two isotropic elasticity tensors that may lead to loss of ellipticity through periodic homogenization, cf.~\cite{Gutierrez-98,Briane-Francfort-15}.
Let us state our main results in the periodic setting first, in which case $\LL=\chi \LL^1+(1-\chi)\LL^2$, where $\chi$
is the  $[0,1)^2$-periodic characteristic function of the strong phase. We assume for simplicity that the level sets of $\LL$ are locally flat (see Section \ref{sec:ishom} for more general geometric assumptions),
and consider three classes of microstructures:
\begin{itemize}
\item[Class~A:] The weak phase $\LL^2$ is  connected in $\R^2$, that is $\{\chi=0\}$ is connected;
\item[Class~B:] The strong phase $\LL^1$ is connected in $\R^2$,  that is $\{\chi=1\}$ is connected;
\item[Class~C:] Neither of the phases are connected in $\R^2$.  
\end{itemize}
Examples of such microstructures are displayed on Figure \ref{fig.structures} below.
\begin{figure}[ht]
	\includegraphics[scale=0.7]{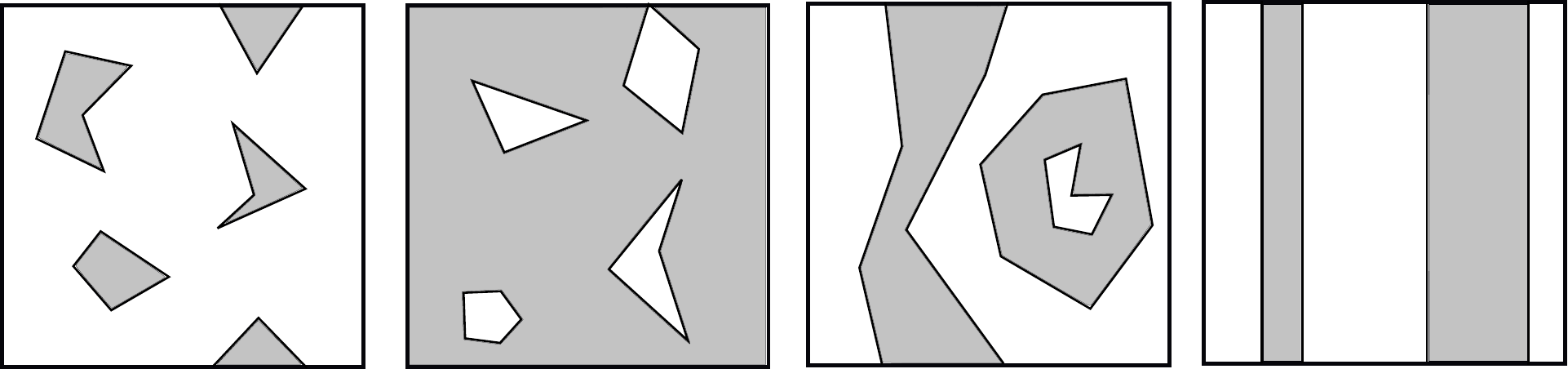}%
	\caption{Periodic microstructures. The strong phase is shaded in gray. From left to right: Class~A, Class~B, and Class~C (non-laminate, laminate).}
\end{figure}\label{fig.structures}

Our main result states that among these three general classes of microstructures (up to a technical geometric assumption on the level-sets, which we
 need for B and C to establish the solvability of an overdetermined elliptic equation on $\R^2$, cf.~\eqref{eq:THEPDE}), only laminate structures (thus a specific subclass of class C) with volume
fractions $\theta_1=\theta_2=\frac12$ (that is, precisely Guti\'errez' example) lead to loss of ellipticity, cf.~Table~\ref{e.table-per}.
\begin{center}
\begin{table}[ht]
\begin{tabular}{|l|c|c|}
\hline
& $\theta_1\ne \theta_2$ & $\theta_1=\theta_2=\frac12$
\\
\hline
Class A & no loss& no loss
\\
\hline
Class B& no loss& no loss
\\
\hline
Class C -- laminate& no loss& \textbf{loss}
\\
\hline
Class C -- non-laminate& no loss& no loss
\\
\hline
\end{tabular}

\bigskip

\caption{\emph{Loss} versus \emph{no loss} of ellipticity through periodic homogenization depending on the geometry of the microstructure and the volume fractions of weak and strong phases.}
\end{table}\label{e.table-per}
\end{center}

We perform our analysis in the general stationary ergodic setting (which is significantly more general than periodic microstructures and requires us to slightly revisit the linear theory  \cite{GMT-93}). In this generality, our results are not as definite as in Table~\ref{e.table-per} (due to a drastic lack of compactness).
The rest of this article is organized as follows.
In Section~\ref{sec:sthom} we recall basic facts on stochastic homogenization, and extend the elegant theory of Geymonat, M\"uller and Triantafyllidis \cite{GMT-93} to the random setting (by introducing a notion of random Bloch wave decomposition). The proofs are displayed in Section~\ref{sec:GMT}.
Section~\ref{sec:ishom} is dedicated to the main results of the paper: the estimates of the ellipticity constants for a mixture of Guti\'errez' isotropic
phases for various microscopic geometries, the proofs of which are postponed to Section~\ref{sec:proofs-ell}. 

\tableofcontents


\section{Stationary GMT theory}\label{sec:sthom}

We refer the reader to \cite[Chapter~12]{JKO} and \cite{DalMaso-Modica-86} for standard background on stochastic homogenization in linear elasticity.
We let $(\Omega,\F,\P)$ be a probability space, where we see $\Omega$ as the set of uniformly bounded measurable elasticity tensor fields $\omega:\R^d\to \Mdd$ (set of symmetric fourth-order tensors).
We assume that our measure $\P$ (which characterizes the microstructures) is invariant by integer shift, which means that for any $A\in \F$,
and all $z\in \Z^d$, $\P(A)=\P(T_zA)$ where $T_z(A)=\{\omega(\cdot+z)\,|\omega \in A\}$. We also assume that $\P$ is ergodic in the sense that if for some $A\in \F$
we have $T_z A\subset A$ for all $z\in \Z^d$, then $\P(A)=0$ or $\P(A)=1$.

\medskip

We then let $\LL$ be a random field distributed according to $\P$ --- note that $\LL$ is simply another name for $\omega$.
We use the standard redundant notation $\LL:\R^d\times \Omega \to \Mdd,(x,\omega)\mapsto \LL(x,\omega):=\omega(x)$ for the field,
and the standard notation $\LL(\cdot+z,\omega)=\LL(\cdot,\omega(\cdot+z))=\LL(x,T_z\omega)$ for the shifted version of $\LL$ by $z\in \Z^d$.

\medskip

A random field $v:\R^d\times \Omega \to \R^n$ ($n\in \N$) is a jointly measurable function of $\omega$ and $x$ (or equivalently of $\LL$ and $x$).
We say that $v$ is stationary if for all translations $z\in \Z^d$, we have $v(\cdot+z,\omega)=v(\cdot,T_z\omega)$, $\P$-almost surely.

\medskip

We start by recalling the standard stochastic homogenization result under the very strong ellipticity assumption.
To this aim, we define Hilbert spaces of (jointly measurable) stationary functions
\begin{eqnarray*}
\Hhr^0_1&:=&\Big\{ v\in L^2_\loc(\R^d,L^2(\Omega,\R^d))\,\Big| \, v(x+z,\omega)=v(x,T_{z}\omega) \\
&&\qquad\qquad\qquad\qquad\qquad\qquad \ \text{for a.e. }\omega\in \Omega, x \in \R^d,z\in \Z^d\Big\},
\\
\Hhr^1_1&:=&\Big\{ v\in H^1_\loc(\R^d,L^2(\Omega,\R^d))\,\Big| \, v(x+z,\omega)=v(x,T_{z}\omega) \\
&&\qquad\qquad\qquad\qquad\qquad\qquad\ \text{for a.e. }\omega\in \Omega, x \in \R^d,z\in \Z^d\Big\},
\end{eqnarray*}
endowed with the norms
\begin{equation*}
\|v\|_{\Hhr^0_1}^2\,=\, \expec{\int_{Q}|v(x,\cdot)|^2dx},\quad
\|v\|_{\Hhr^1_1}^2\,=\, \expec{\int_{Q}\left(|v(x,\cdot)|^2+|\nabla v(x,\cdot)|^2\right)dx},
\end{equation*}
where here (and in what follows) we denote by $Q_R=(-R/2,R/2)^d$ a cube with side-length $R$ and $Q=Q_1$.
The above spaces coincide with $L^2(Q)$ and $H^1_\per(Q)$ in the case when $\P$ charges $Q$-periodic maps. 
\begin{theorem}\label{t.JKO-standard}
Let $D$ be a bounded Lipschitz domain of $\R^d$.
Assume that $\P$ only charges very strongly elliptic tensor fields in the sense that there exists $0<\lambda\le 1$ such that for all $M\in \Msym$
and for almost every $x\in \R^d$ we have
\begin{equation}\label{e.ell-GMT}
\lambda |M|^2 \le M\cdot \LL(x,\omega)M\,\le \, |M|^2.
\end{equation}
Fix $u_0\in H^1(D)$ and set $\K:=\{u+u_0, u\in H^1_0(D)\}$.
Then the (random) integral functional $\I_\e:\K \to \R, u\mapsto \int_D \nabla u\cdot \LL(\frac{\cdot}{\e},\omega)\nabla u$ $\Gamma(L^2)$-converges to 
the (deterministic) integral functional $\I_*:\K \to \R, u\mapsto \int_D \nabla u\cdot \LL_*\nabla u$ as $\e \downarrow 0$ almost surely, where $\LL_*$ is a very strongly elliptic constant elasticity tensor characterized
by 
\begin{equation}\label{e.asfo-JKO}
M\cdot \LL_* M\,=\,\inf_{v \in \Hhr^1_1} \expec{\int_Q (M+\nabla v)\cdot \LL (M+\nabla v)}.
\end{equation}
\qed
\end{theorem}
\begin{remark}
If $\P$ is a Dirac measure on some periodic tensor field $\LL$ then \eqref{e.asfo-JKO} is the standard (periodic) cell-formula. 
Although the minimizer is attained at a periodic deformation field $u_M$ for a periodic tensor field $\LL$, it is not attained in $\Hhr^1_1$ in general for (genuinely) random tensor fields (due to the lack of compactness related to the absence of Poincar\'e inequality in $\Hhr^1_1$ --- indeed, the latter would essentially imply a Poincar\'e inequality 
in $H^1(D)$ with a constant independent of the size of $D$,  contradicting the scaling of the inequality).
\qed
\end{remark}

We now turn to the extension of the GMT theory (cf. \cite[Section~3]{GMT-93}) to this random setting.
We no longer assume that $\P$ only charges \emph{very} strongly tensor fields, but only require the lower bound in \eqref{e.ell-GMT} to
hold for rank-one matrices $M=a\otimes b$ (the upper-bound is unchanged). 
We introduce the following main measures of coercivity: 
\begin{eqnarray*}
\Lambda&=&\expec{\inf_{u\in C^\infty_0(\R^d,\R^d)}\frac{\int_{\R^d}{\nabla u}\cdot \LL\nabla u}{\int_{\R^d}|\nabla u|^2}},
\\
\Lambda_4&=& \inf\left\{ \frac{\expec{\int_Q ({a\otimes b+\nabla v}) \cdot \LL (a\otimes b+\nabla v)}}{\expec{\int_Q |a\otimes b +\nabla v |^2}}\,\bigg|\, a,b\in \R^d,v\in \Hhr^1_1\right\},\\
\Lambda_6&=& \inf\left\{ \frac{\expec{\int_Q {\nabla v} \cdot \LL \nabla v}}{\expec{\int_Q |\nabla v |^2}}\,\bigg|\,v\in \Hhr^1_1\right\}.
\end{eqnarray*}
\begin{remark}
In the periodic setting considered in \cite{GMT-93}, the corresponding ellipticity constants are recovered by dropping the expectations, and replacing $\Hhr^1_1$
by $H^1_\per(Q,\R^d)$.
\qed
\end{remark}
In \cite[Lemma 4.2]{GMT-93} the authors prove several characterizations of $\Lambda$. In the stationary ergodic setting, a similar analysis is possible, which we include for the sake of completeness. To this end we first extend the real-valued spaces $\Hhr$ to complex-valued (jointly measurable) $N$-stationary function spaces: For all $N\in \N$ we set
\begin{eqnarray*}
	\Hh^0_N&:=&\Big\{ v\in L^2_\loc(\R^d,L^2(\Omega,\C^d))\,\Big| \, v(x+Nz,\omega)=v(x,T_{Nz}\omega) \\
	&&\qquad\qquad\qquad\qquad\qquad\qquad \ \text{for a.e. }\omega\in \Omega, x \in \R^d,z\in \Z^d\Big\},
	\\
	\Hh^1_N&:=&\Big\{ v\in H^1_\loc(\R^d,L^2(\Omega,\C^d))\,\Big| \, v(x+Nz,\omega)=v(x,T_{Nz}\omega) \\
	&&\qquad\qquad\qquad\qquad\qquad\qquad\ \text{for a.e. }\omega\in \Omega, x \in \R^d,z\in \Z^d\Big\},
\end{eqnarray*}
endowed with the norms
\begin{equation*}
\|v\|_{\Hh^0_N}^2\,=\, \expec{\int_{Q_N}|v(x,\cdot)|^2dx},\quad
\|v\|_{\Hh^1_N}^2\,=\, \expec{\int_{Q_N}\left(|v(x,\cdot)|^2+|\nabla v(x,\cdot)|^2\right)dx}.
\end{equation*}
Then we introduce the following additional coercivity constants:
\begin{eqnarray*}
	\Lambda_1&=& \inf\left\{ \frac{\expec{\int_Q \bar{\nabla v} \cdot \LL \nabla v}}{\expec{\int_Q |\nabla v |^2}}\,\bigg|\, v(x,\omega)=e^{i\gamma\cdot x}p(x,\omega),\gamma\in [0,2\pi)^d,
	p\in \Hh^1_1\right\},\\
	\Lambda_2&=&\inf\left\{ \frac{\expec{\int_{Q_N} \bar{\nabla v} \cdot \LL \nabla v}}{\expec{\int_{Q_N} |\nabla v |^2}}\,\bigg|\, N\geq 1,v(x,\omega)=e^{i\gamma\cdot x}p(x,\omega),\gamma\in \R^d,
	p\in \Hh^1_N\right\},\\
	\Lambda_3&=&\inf\left\{ \frac{\expec{\int_{Q_N} \bar{\nabla v} \cdot \LL \nabla v}}{\expec{\int_{Q_N} |\nabla v |^2}}\,\bigg|\, N\geq 1,v\in \Hh^1_N\right\},
	\\
	\Lambda_5&=&\liminf_{\gamma\downarrow 0}\quad \inf\left\{ \frac{\expec{\int_{Q} \bar{\nabla v} \cdot \LL \nabla v}}{\expec{\int_{Q} |\nabla v |^2}}\,\bigg|\, v(x,\omega)=e^{i\gamma\cdot x}p(x,\omega),
	p\in \Hh^1_1\right\}.\\
\end{eqnarray*}
The next result is the analogue to \cite[Lemma 4.2]{GMT-93}. Its proof relies on a stochastic version of the Bloch wave transform.
\begin{lemma}\label{l.Lambda123}
	If $\Lambda\geq 0$, then $\Lambda=\Lambda_1=\Lambda_2=\Lambda_3$.\qed
\end{lemma}
The following result is the counterpart of \cite[Theorem~3.3(i)]{GMT-93} for random coefficients. 
\begin{proposition}\label{p.order}
If $\Lambda\ge 0$, then we have $\Lambda_6\ge \Lambda_4\ge \Lambda$.
\qed
\end{proposition}
\begin{remark}\label{r.Lambda45}
	In the proof of \pref{order}  we show the inequality $\Lambda_4\ge \Lambda_5$. In the periodic setting we have in addition  $\Lambda_4=\Lambda_5$. The proof
	of this identity in \cite{GMT-93} uses however Poincar\'e's inequality, which does not hold in the present random setting ---  this identity is not needed for our results.
	\qed	
\end{remark}
Before we turn to the homogenization properties, let us give an interpretation of the main ellipticity constants:
\begin{itemize}
\item $\Lambda$ measures the global coercivity of the nonhomogeneous tensor $\LL$. 
\item $\Lambda_4$ measures coercivity with respect to shearing deformations (modulo $\Z^d$-stationary contributions).
\item $\Lambda_6$ measures coercivity with respect to $\Z^d$-stationary, possibly highly localized deformations.
\end{itemize}
The following result is a strict generalization of \cite[Theorem~3.4]{GMT-93}:
\begin{theorem}\label{t.alaGMT}
Let $D$ be a Lipschitz bounded domain, $u_0\in H^1(D)$, and $\K:=\{u=u_0+v,v\in H^1_0(D)\}$.
For all $\e>0$, define $\I_\e:\K\to \R, u\mapsto \int_D \nabla u\cdot \LL(\frac{\cdot}{\e})\nabla u$.
If $\Lambda\ge 0$ and $\Lambda_6>0$, then $\I_\e$ $\Gamma(w-H^1)$-converges on $\K$ to $\I_*:\K\to \R,u\mapsto \int_D \nabla u\cdot \LL_*\nabla u$
almost surely, where $\LL_*$ is given in direction $M\in \Md$ by
$$
M \cdot \LL_*M\,:=\,\inf_{u\in \Hhr^1_1} \expec{\int_Q (M+\nabla u)\cdot \LL(M+\nabla u)}.
$$
In addition, if $\Lambda_4>0$, then $\LL_*$ is strictly strongly elliptic, whereas if $\Lambda_4=0$ then $\LL_*$ loses strict strong ellipticity and there exists
a rank-one matrix $a\otimes b\in \Md$ such that $a\otimes b \cdot \LL_* a\otimes b=0$.
\qed
\end{theorem}
\begin{remark}\label{r.BraidesBriane}
Braides and Briane proved the $\Gamma$-convergence result of \tref{alaGMT} under the sole assumption
of $\Lambda\ge 0$ using soft arguments (see \cite[Theorem 2.4]{Briane-PallaresMartin}). Although their result is stated only in the periodic setting, the proof easily extends to the random stationary ergodic setting.  This remark will be used when we cannot 
establish the bound $\Lambda_6>0$.
\qed
\end{remark}
\begin{remark}
In the recent preprint \cite{Briane-Francfort-18}, Briane and Francfort improved \tref{alaGMT} 
by replacing the $\Gamma$-convergence result for the weak topology of $H^1(D)$ by a $\Gamma$-convergence result
for the weak $L^2(D)$-topology under two assumptions: the microstructure is periodic (two-scale convergence is used in the proof)
and is of class A or C (laminate) --- cf.~the introduction.
\qed
\end{remark}
%


\section{Stochastic homogenization, connectedness, and strong ellipticity} \label{sec:ishom}

In this section we investigate the influence of the geometry of the composite material on the strong ellipticity of the homogenized stiffness tensor.
We only treat the case of dimension $d=2$
and give an example of a class of mixtures for which we can prove that there is homogenization ($\Lambda\ge 0$), and for which the associated homogenized stiffness tensor $\LL_*$ loses ellipticity essentially only for a laminate structure.
Throughout this article, we 
let $\lambda_1,\mu_1$ and $\lambda_2,\mu_2$ be the
Lam\'e coefficients of isotropic stiffness tensors $\LL^1$ and $\LL^2$ that satisfy
\begin{equation}
\label{eq:cond}
0<\mu_1=-(\lambda_2+\mu_2)<\mu_2, \lambda_1+\mu_1>0.
\end{equation}
Recall that elasticity tensors are characterized (for $i=1,2$) by their actions as quadratic forms on matrices $A=(a_{ij})_{ij}$,
which for isotropic tensors read
\begin{equation}\label{a+5}
A\cdot{\LL^i}A=(\lambda_i+2\mu_i) (a_{11}^2+a_{22}^2)
+2\lambda_i a_{11}a_{22}+\mu_i (a_{12}+a_{21})^2.
\end{equation}
We also define the average volume fractions of the two phases by $\theta_1 = \expec{\big|\{x\in Q\,|\,\LL(x)=\LL^1\}\big|}$ and $\theta_2=1-\theta_1$. Throughout this paper we assume $\theta_1\in (0,1)$.

We first consider the case of inclusions of the strong phase $\LL^1$ in the weak matrix phase $\LL^2$.
\begin{assumption}
\label{a.non-neg}
Let the random field $\LL:\R^2 \times \Omega \to \{\LL^1,\LL^2\}$ be measurable and stationary, and have the following properties: For almost every realization,
\begin{itemize}
	\item[(A1)] the set $\{x\,|\,\LL(x)= \LL^2\}$ is open and connected (that is, the microstructure is ``inclusion"-like) with a Lipschitz-regular boundary;
	\item[(A2)] there exists a constant $C_1<\infty$ such that ${\rm diam}(I)\leq C_1$ for each connected component $I$ of $\{x\,|\,\LL(x)=\LL^1\}$;
	\item[(A3)] there exists a constant $C_2<\infty$ such that for all $R>0$ large enough and all $u\in L^2(M_R)$ with $\int_{M_R}u=0$ we have
	the Ne\v{c}as inequality
	\begin{equation}\label{e.Neceq}
	\|u\|^2_{L^2(M_R)}\leq C_2 \|\nabla u\|^2_{H^{-1}(M_R)},
	\end{equation}
	where  $Q^1_R=(-R/2-2C_1,R/2+2C_1)^d$ denotes an enlarged cube, and where
	\begin{equation*}
	M_R=Q^1_R\setminus \bigcup_{I_j\cap [-R/2,R/2)^d\neq\emptyset}I_j,
	\end{equation*}
	with $I_j$ denoting the connected components of $\{x\,|\,\LL(x)=\LL^1\}$.
\end{itemize}
\qed
\end{assumption}
\begin{remark}\label{r.onA3}
The Ne\v{c}as inequality \eqref{e.Neceq} holds true for any bounded connected Lipschitz domain and (in contrast to the Poincar\'e inequality) it is invariant under dilation of the domain. In order for \eqref{e.Neceq} to hold on perforated domains like $M_R$, it is enough to have an extension operator $\mathcal{E}_R$ from $L^2(M_R)/\R$ to $L^2(Q^1_R)/\R$ such that $\|\nabla\mathcal{E}_R(u)\|_{H^{-1}(Q^1_R)}\leq C\|\nabla u\|_{H^{-1}(M_R)}$ with $C$ independent of $R$.
We refer the reader to the appendix for sufficient geometric conditions that ensure the validity of  (A3).
\qed
\end{remark}
We now consider the converse situation when $\LL^1$ is the matrix, and $\LL^2$ are the inclusions.
\begin{assumption}\label{b.non-neg}
Let the random field $\LL:\R^2 \times \Omega\to \{\LL^1,\LL^2\}$ be measurable and stationary, and have the following properties: For almost every realization,
\begin{itemize}
	\item[(B1)] the set $\{x\,|\,\LL(x)= \LL^1\}$ is open and connected (that is, the microstructure is ``inclusion"-like) with a Lipschitz-regular boundary;
	\item[(B2)] there exists a constant $C_1<\infty$ such that ${\rm diam}(I)\leq C_1$ for each connected component $I$ of $\{x\,|\,\LL(x)=\LL^2\}$;
	\item[(B3)] the boundary of the set $\{x\,|\,\LL(x)= \LL^1\}$ contains a flat segment or an arc of circle.
\end{itemize}
\qed
\end{assumption}
Finally, we consider the situation when neither $\LL^1$ nor $\LL^2$ are connected. Our assumptions then read:
\begin{assumption}\label{c.non-neg}
Let the random field $\LL:\R^2 \times \Omega\to \{\LL^1,\LL^2\}$ be measurable and stationary, and have the following properties: For almost every realization,
\begin{itemize}
	\item[(C1)] the set $\{x\,|\,\LL(x)= \LL^1\}$ is open and not connected, with a Lipschitz-regular boundary;
	\item[(C2)] the set $\{x\,|\,\LL(x)= \LL^2\}$ is open and not connected, with a Lipschitz-regular boundary;
\end{itemize}
and one of the following conditions holds:
\begin{itemize}
\item[(C3-a)] there exists a connected components of $\{x\,|\,\LL(x)= \LL^1\}$ which is not an infinite stripe and which contains a flat segment in its boundary;
\item[(C3-b)] all the connected components of $\{x\,|\,\LL(x)= \LL^1\}$  and $\{x\,|\,\LL(x)= \LL^2\}$ are parallel bands
orthogonal to $e_1$.
\end{itemize}
\qed
\end{assumption}
We start our analysis with an elementary observation.
\begin{proposition}\label{p.LAMBDA}
If $\LL: \R^2 \times \Omega\to \{\LL^1,\LL^2\}$ is measurable and stationary, then $\Lambda\geq 0$.
In particular, by \rref{BraidesBriane}, homogenization holds, and $\LL_*$ is well-defined. We denote by $\Lambda_*\ge 0$ the best ellipticity constant
for $\LL_*$.
\qed	
\end{proposition}
We first treat geometries of type~\ref{a.non-neg}.
The following result is a strict extension of \cite[Theorem~2.9]{Briane-Francfort-15}  which relaxes most of the geometric assumptions on $\LL$ .
\begin{proposition}\label{p.lambdaA}
Assume that  the random field $\LL: \R^2 \times \Omega\to \{\LL^1,\LL^2\}$ satisfies Assumption~\ref{a.non-neg}.
Then $\Lambda_6\ge \Lambda_4>0$.
\qed
\end{proposition}
The combination of  \pref{lambdaA} and \tref{alaGMT} yields the first main result of this article:
\begin{corollary}\label{c.hom-homA}
Assume that the random field $\LL: \R^2 \times \Omega\to \{\LL^1,\LL^2\}$ satisfies Assumption~\ref{a.non-neg}.
Then  the homogenized stiffness tensor $\LL_*$ is \emph{strictly} strongly elliptic, i.~e. $\Lambda_*>0$.
\qed
\end{corollary}
We turn now to geometries of type~\ref{b.non-neg}.
The result below, which corresponds to \pref{lambdaA}, is weaker in two respects: first it only covers the case of periodic configurations (rather than stationary 
ergodic), and second our argument only applies to specific shapes of inclusions  satisfying (B3) (we expect this issue is only technical, and the result to be generic).
\begin{proposition}\label{p.lambdaB}
Assume that the random field  $\LL:\R^2\times \Omega \to \{\LL^1,\LL^2\}$ satisfies Assumption~\ref{b.non-neg}.
Then we have the implication $\Lambda_6>0 \implies \Lambda_4>0$.
If in addition, $\LL$ is periodic, then $\Lambda_6>0$ (we do not presently know whether this holds in the random setting too).
\qed
\end{proposition}
The combination of  \pref{lambdaB} and \tref{alaGMT} yields the second main result of this article:
\begin{corollary}\label{c.hom-homB}
Assume that the random field  $\LL:\R^2\times \Omega \to \{\LL^1,\LL^2\}$ satisfies Assumption~\ref{b.non-neg}
and that $\LL$ is periodic.
Then the homogenized stiffness tensor $\LL_*$ is well-defined, and it is \emph{strictly} strongly elliptic, i.~e. $\Lambda_*>0$.
\qed
\end{corollary}
\begin{remark}\label{r.onShape} 
If there is only one single connected component in the periodic cell, then the result holds whatever the shape of the inclusion, that is to say, without Assumption (B3).
\qed
\end{remark}
We conclude this section with an extension of the result by Briane and Francfort
when neither of the two phases are connected.
\begin{proposition}\label{p.lambdaC}
Assume that the random field  $\LL:\R^2\times \Omega\to \{\LL^1,\LL^2\}$ satisfies Assumption~\ref{c.non-neg}.
If $\LL$ satisfies (C3-a) then we have the implication $\Lambda_6>0 \implies \Lambda_4>0$, and $\Lambda_6>0$ under the additional assumption that $\LL$ is periodic (we do not presently know whether this holds in the random setting too).
If $\LL$ satisfies (C3-b), then
\begin{itemize}
\item either $\LL$ has volume fractions $\theta_1=\theta_2=\frac12$, in which case $\Lambda_4=0$;
\item or $\LL$ has volume fractions $\theta_1\ne \frac12\ne\theta_2$, in which case we have the implication $\Lambda_6>0 \implies \Lambda_4>0$.
\end{itemize}
\qed
\end{proposition}
Under assumption (C3-b) we do not know whether $\Lambda_6>0$ in the random setting. However, when $\theta_1\ne \frac12$ we can directly rule out the loss of ellipticity, which has to be understood in the sense of \rref{BraidesBriane}. This is part of the last main result of this article:
\begin{corollary}\label{c.hom-homC}
Assume that the random field  $\LL:\R^2\times \Omega \to \{\LL^1,\LL^2\}$ satisfies Assumption~\ref{c.non-neg}.
If $\LL$ satisfies (C3-a) and is periodic, then  $\LL_*$ is \emph{strictly} strongly elliptic, i.~e. $\Lambda_*>0$.
If $\LL$ satisfies (C3-b), then 
\begin{itemize}
\item either $\LL$ has volume fractions $\theta_1=\theta_2=\frac12$, in which case $\LL_*$ is \emph{not strictly} strongly elliptic, i.~e. $\Lambda_*=0$;
\item or $\LL$ has volume fractions $\theta_1\ne\frac12\ne\theta_2$, in which case $\LL_*$ is  \emph{strictly} strongly elliptic, i.~e. $\Lambda_*>0$.
\end{itemize}
\qed
\end{corollary}
Figure \ref{f.fig211} below is a summary of the main results.
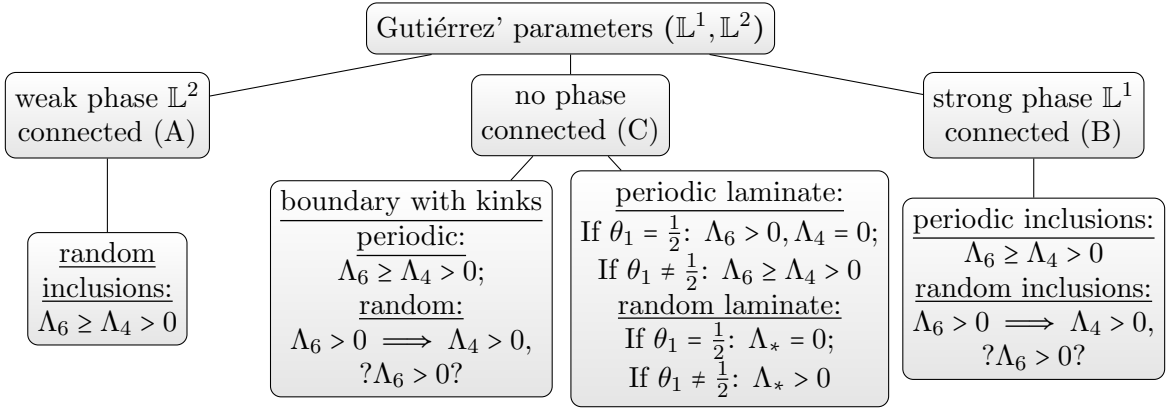
\begin{figure}[h]
\centering
\begin{tikzpicture}[level 1/.style={sibling distance=16em, level distance=3em},level 2/.style={sibling distance=11em, level distance=6em},
every node/.style = {shape=rectangle, rounded corners,
	draw, align=center,
	top color=white, bottom color=gray!20}]]
\node {Guti\'errez' parameters $(\LL^1,\LL^2)$}
child { node {weak phase $\LL^2$ \\ connected (A)} 
	child { node {\underline{random} \\ \underline{inclusions:} \\ $\Lambda_6\ge \Lambda_4>0$}}}
child { node {no phase \\ connected (C)}
	child { node {\underline{boundary with kinks} \\ \underline{periodic:} \\ $\Lambda_6\ge \Lambda_4>0$; \\ \underline{random:} \\ $\Lambda_6>0 \implies \Lambda_4>0$, \\  ?$\Lambda_6>0$?}}
	child { node {\underline{periodic laminate:} \\ If $\theta_1=\frac12$:   $\Lambda_6>0,\Lambda_4=0$; \\ If $\theta_1\ne \frac12$: 
	  $\Lambda_6\ge \Lambda_4>0$  \\ {\underline{random laminate:}} \\ If $\theta_1=\frac12$:  $\Lambda_*=0 $; \\ If $\theta_1\ne \frac12$: 
	  $\Lambda_*>0$  }}}
child { node {strong phase $\LL^1$ \\ connected (B)}
	child { node {\underline{periodic inclusions:} \\  $\Lambda_6\ge \Lambda_4>0$ 
	\\ \underline{random inclusions:} \\  $\Lambda_6>0 \implies \Lambda_4>0$, \\  ?$\Lambda_6>0$?}}};
\end{tikzpicture}
\caption{Ellipticity constants $\Lambda_6$, $\Lambda_4$, and $\Lambda_*$ of the homogenized tensor $\LL_*$ for two-phase composites $(\LL^1,\LL^2)$
with volume fractions $\theta_1>0$ and $\theta_2=1-\theta_1>0$.}
\label{f.fig211}
\end{figure}

\section{Proofs: Stationary GMT theory}\label{sec:GMT}

\subsection{Proofs of \lref{Lambda123} and \pref{order}}

We split the proof into seven steps.

\medskip

\noindent \textit{Step~1.} Reformulations.

\noindent Since $\LL$ is a symmetric real-valued 4-tensor, we have the following equivalent definitions for $\Lambda, \Lambda_4,\Lambda_6$ using complex-valued fields:
\begin{eqnarray*}
\Lambda&=&\expec{\inf_{u\in C^\infty_0(\R^d,\C^d)}\frac{\int_{\R^d}\bar{\nabla u}\cdot \LL\nabla u}{\int_{\R^d}|\nabla u|^2}},
\\
\Lambda_4&=& \inf\left\{ \frac{\expec{\int_Q (\bar{a\otimes b+\nabla v}) \cdot \LL (a\otimes b+\nabla v)}}{\expec{\int_Q |a\otimes b +\nabla v |^2}}\,\bigg|\, a\in \C^d,b\in \R^d,v\in \Hh^1_1\right\},\\
\Lambda_6&=& \inf\left\{ \frac{\expec{\int_Q \bar{\nabla v} \cdot \LL \nabla v}}{\expec{\int_Q |\nabla v |^2}}\,\bigg|\,v\in \Hh^1_1\right\}.
\end{eqnarray*}
The inequalities $\Lambda_6\ge \Lambda_4$ and $\Lambda_5\ge \Lambda_1$ are obvious. It is therefore enough to prove that
$\Lambda= \Lambda_1= \Lambda_2= \Lambda_3$ on the one hand, and $\Lambda_4\ge \Lambda_5$ on the other hand.

\medskip

\noindent \textit{Step~2.} Stochastic Bloch wave transform.

\noindent For all $v \in L^2(\Omega,C^\infty_{0}(\R^d,\R^d))$ such that $v$ has support in some fixed domain $B_R$ almost surely, we define the Bloch transform of $v$ as follows: for all $\gamma \in [0,2\pi)^d$,
$\tilde v_\gamma:\Omega\times \R^d\to \C^d$ is given by
\begin{equation*}
\tilde v_\gamma(x,\omega)\,:=\, \sum_{z\in \Z^d} e^{-i\gamma\cdot z}v(x+z,T_{-z}\omega),
\end{equation*}
where the sum is finite since $x\mapsto v(x,\omega)$ has support in $B_R$ almost surely.
The interest of the Bloch transform is that it maps fields with compact support onto stationary fields (up to a phase).
Indeed, for all $x\in \R^d$, $y\in \Z^d$, $\gamma \in [0,2\pi)^d$, and almost every $\omega\in \Omega$, we have
using the group property of $\{T_z\}_{z\in \Z^d}$ 
\begin{eqnarray*}
\tilde v_\gamma(x+y,\omega)&=&\sum_{z\in \Z^d} e^{-i\gamma\cdot z}v(x+y+z,T_{-z}\omega) \\
&=&e^{i\gamma\cdot y}\sum_{z\in \Z^d} e^{-i\gamma\cdot (y+z)}v(x+y+z,T_{-(y+z)}T_y \omega) \\
&=&e^{i\gamma\cdot y}\tilde v_\gamma(x,T_y \omega),
\end{eqnarray*}
so that $\hat v_\gamma :(x,\omega) \mapsto e^{-i\gamma \cdot x}\tilde v_\gamma(x,\omega)$ is a stationary field.

\medskip

\noindent \textit{Step~3.} Proof of $\Lambda\ge \Lambda_1$.

\noindent
As shown in Appendix~\ref{append:selection}, $\inf_{C^\infty_0(\R^d,\R^d)} \frac{\int_{\R^d} \nabla v\cdot \LL \nabla v}{\int_{\R^d} |\nabla v|^2}$
is measurable on $(\Omega,\mathcal F)$, and therefore constant by stationarity.
For all $L>0$, denote by $C^\infty_{0,L}(\R^d,\R^d)$ the subset of compactly supported smooth functions the support of which 
is contained in the ball $B_L$ centered at the origin and of radius $L$. 
As above, for all $L\ge 0$, $\inf_{C^\infty_{0,L}(\R^d,\R^d)} \frac{\int_{\R^d} \nabla v\cdot \LL \nabla v}{\int_{\R^d} |\nabla v|^2}$
is measurable on $(\Omega,\mathcal F)$, and the monotone convergence theorem implies that for all $\delta>0$ there exists $L_\delta\ge 0$ such that
$$
\expec{\inf_{C^\infty_{0,L_\delta}(\R^d)} \frac{\int_{\R^d} \nabla v\cdot \LL \nabla v}{\int_{\R^d} |\nabla v|^2}} \le \expec{\inf_{C^\infty_{0}(\R^d)} \frac{\int_{\R^d} \nabla v\cdot \LL \nabla v}{\int_{\R^d} |\nabla v|^2}}+\frac\delta2=\Lambda+\frac\delta 2.
$$
From the measurable selection argument of Appendix~\ref{append:selection}, we thus deduce the existence of measurable
quasi-minimizers $v_\delta \in L^2(\Omega,C^\infty_{0,L_\delta}(\R^d,\R^d))$ such that 
 almost surely $\int_{\R^d} |\nabla v_\delta|^2=1$ (by scaling) and 
\begin{equation}\label{e.meas-selection}
\expec{\int_{\R^d}\bar{\nabla v_\delta}\cdot \LL \nabla v_\delta}\leq \Lambda+\delta.
\end{equation}
For all $\gamma \in [0,2\pi)^d$, we consider the Bloch transform $\tilde v_{\delta,\gamma}$ of $v_{\delta}$ (the support of which is deterministic).
Let $Q^*=[0,2\pi)^d$.
We then have for almost every $\omega\in \Omega$:
\begin{eqnarray*}
&&{\int_Q \int_{Q^*} \bar{\nabla \tilde v_{\delta,\gamma}(x,\omega)}\cdot \LL(x,\omega) \nabla \tilde v_{\delta,\gamma}(x,\omega)d\gamma dx} \\
&=&\int_Q \sum_{z\in \Z^d}\sum_{z'\in \Z^d}\int_{Q^*} e^{-i\gamma\cdot(z'-z)} \bar{\nabla  v_{\delta}(x+z,T_{-z}\omega)}\cdot \LL(x,\omega) \nabla v_{\delta}(x+z',T_{-z'}\omega)d\gamma dx \\
&=&(2\pi)^d \int_Q \sum_{z\in \Z^d} \bar{\nabla v_{\delta}(x+z,T_{-z}\omega)}\cdot \LL(x,\omega) \nabla  v_{\delta}(x+z,T_{-z}\omega)dx\\
&=&(2\pi)^d \int_Q \sum_{z\in \Z^d} \bar{\nabla v_{\delta}(x+z,T_{-z}\omega)}\cdot \LL(x+z,T_{-z}\omega) \nabla  v_{\delta}(x+z,T_{-z}\omega)dx,
\end{eqnarray*}
using that $\int_{Q^*}e^{-i\gamma\cdot(z'-z)}d\gamma =(2\pi)^d \delta_{zz'}$ (the Kronecker symbol), and the stationarity of $\LL$.
Since the translation group is measure preserving, the expectation of this identity turns into
\begin{eqnarray}
\lefteqn{\expec{\int_Q \int_{Q^*} \bar{\nabla \tilde v_{\delta,\gamma}(x,\omega)}\cdot \LL(x,\omega) \nabla \tilde v_{\delta,\gamma}(x,\omega)d\gamma dx}} \nonumber \\
&=& (2\pi)^d \expec{\int_Q \sum_{z\in \Z^d} \bar{\nabla v_{\delta}(x+z,\omega)}\cdot \LL(x+z,\omega) \nabla  v_{\delta}(x+z,\omega)dx} \nonumber \\
&=&(2\pi)^d \expec{\int_{\R^d}\bar{\nabla v_{\delta}}\cdot \LL \nabla  v_{\delta}}.
\label{eq:order-3}
\end{eqnarray}
Likewise, 
\begin{equation}\label{eq:order-3-00}
\expec{\int_Q \int_{Q^*} |\nabla \tilde v_{\delta,\gamma}(x,\omega)|^2d\gamma dx}\,=\,(2\pi)^d \expec{\int_{\R^d}|\nabla v_{\delta}|^2} =(2\pi)^d.
\end{equation}
Since for all $\gamma \in [0,2\pi)^d$, $\tilde v_{\delta,\gamma}(x,\omega)=e^{i\gamma\cdot x}\hat v_{\delta,\gamma}(x,\omega)$ with $\hat v_{\delta,\gamma}\in \Hh^1_1$, by definition of $\Lambda_1$
we have
\begin{equation*}
\int_{Q^*} \expec{\int_Q \bar{\nabla \tilde v_{\delta,\gamma}(x,\omega)}\cdot \LL(x,\omega) \nabla \tilde v_{\delta,\gamma}(x,\omega) dx}d\gamma
\,
\geq \,\Lambda_1 \int_{Q^*}  \expec{\int_Q |\nabla \tilde v_{\delta,\gamma}(x,\omega)|^2dx}d\gamma ,
\end{equation*}
so that \eqref{eq:order-3}, \eqref{eq:order-3-00}, \eqref{e.meas-selection}, and Fubini's theorem yield
\begin{equation*}
\Lambda+\delta \ge \expec{\int_{\R^d}\bar{\nabla v_{\delta}}\cdot \LL \nabla  v_{\delta}} \,\geq \,\Lambda_1 \expec{\int_{\R^d}|\nabla v_{\delta}|^2} =\Lambda_1,
\end{equation*}
from which the desired inequality $\Lambda\geq \Lambda_1$ follows by the arbitrariness of $\delta$.

\medskip

\noindent \textit{Step~4.} Proof of $\Lambda_1\ge \Lambda$.

\noindent This is a standard cut-off argument.
Set $\delta>0$ and choose $w:(x,\omega)\mapsto e^{i\gamma\cdot x}p_\gamma(x)$, $\gamma\in Q^*$, $p_\gamma \in \Hh^1_1$ such that
$\expec{\int_Q \bar{\nabla w}\cdot \LL\nabla w}\le (\Lambda_1+\delta)\expec{\int_Q|\nabla w|^2}.$
By stationarity of $p_\gamma$ and the properties of the Hermitian product, for all $z\in \Z^d$, 
\begin{equation}\label{eq:order-1}
\int_{z+Q} \bar{\nabla w}\cdot \LL\nabla w(\omega)=\int_Q \bar{\nabla w}\cdot \LL\nabla w(T_z\omega)\text{ and }\int_{z+Q}|\nabla w|^2(\omega)=\int_Q|\nabla w|^2(T_z\omega).
\end{equation}
For $k\in \N$, let $\eta_k$ be a cut-off for $Q_{2k}$ in $Q_{2(k+[\sqrt{k}])}$ such that $|\nabla \eta_k|\lesssim 1$. Set $v_k:=\eta_kw$.
Then, for $x\in Q_{2k}$, $\nabla v_k(x)=\nabla w(x)$ and for $x\in Q_{2(k+[\sqrt{k}])}\setminus Q_{2k}$, $|\nabla v_k(x)|\leq C(|w(x)|+|\nabla w(x)|)$ for some universal constant $C$.
We thus obtain the two estimates
\begin{eqnarray*}
|\int_{\R^d}\bar{\nabla v_k}\cdot \LL\nabla v_k-\int_{Q_{2k}}\bar{\nabla w}\cdot \LL \nabla w|&\leq & C\int_{Q_{2(k+[\sqrt{k}])}\setminus Q_{2k}}|\nabla w|^2+|w|^2,\\
\int_{\R^d}|\nabla v_k|^2&\geq &\int_{Q_{2k}}|\nabla w|^2.
\end{eqnarray*}
In view of \eqref{eq:order-1}, the ergodic theorem yields almost surely
\begin{eqnarray*}
\lim_{k\uparrow \infty}\fint_{Q_{2k}}\bar{\nabla w}\cdot \LL \nabla w&=&
\expec{\int_Q\bar{\nabla w}\cdot \LL \nabla w},\\
\lim_{k\uparrow \infty} \fint_{Q_{2k}}|\nabla w|^2&=&\expec{\int_Q|\nabla w|^2},
\\
\lim_{k\uparrow \infty} \fint_{Q_{2k}}|w|^2&=&\expec{\int_Q|w|^2},
\end{eqnarray*}
and therefore, since $\lim_{k\uparrow \infty} \frac{|Q_{2k}|}{|Q_{2(k+[\sqrt{k}])}|}=1$, 
\begin{equation*}
\lim_{k\uparrow\infty}(2k)^{-d} \int_{Q_{2(k+[\sqrt{k}])}\setminus Q_{2k}}|\nabla w|^2+|w|^2
\,=\,\lim_{k\uparrow\infty} \Big|\fint_{Q_{2(k+[\sqrt{k}])}}|\nabla w|^2+|w|^2-\fint_{Q_{2k}}|\nabla w|^2+|w|^2\Big|\,=\,0.
\end{equation*}
This implies that almost surely
$$
\liminf_{k\uparrow \infty} \frac{\int_{\R^d}\bar{\nabla v_k}\cdot \LL\nabla v_k}{\int_{\R^d}|\nabla v_k|^2} \,\le  \,\frac{\expec{\int_Q \bar{\nabla w}\cdot \LL\nabla w}}{\expec{\int_Q|\nabla w|^2}}\,\leq\,\Lambda_1+\delta,
$$
and proves the claim by the arbitrariness of $\delta>0$.

\medskip

\noindent \textit{Step~5.} Proof of $\Lambda_2\ge \Lambda_1$.

\noindent  This is analogous to Step~2. Let $N\in \mathbb{N}$, $\gamma'\in\R^d$, and $q_N\in \Hh^1_N$. 
We consider $v:\R^d\to \R$ defined by $v:x\mapsto  e^{i \gamma'\cdot x}q_N(x)$. 
We then apply a variant of the Bloch transform defined in Step~1: For all $\gamma \in I:=\{0,2\pi\frac{1}{N},\dots,2\pi\frac{N-1}{N}\}^d$, we set
\begin{equation*}
\tilde{v}_\gamma(x,\omega)=\sum_{z\in \{0,\dots,N-1\}^d}e^{-i (\gamma+\gamma')\cdot z}v(x+z,T_{-z}\omega).
\end{equation*}

\medskip
\noindent Let now $\tilde{\gamma} \in [0,2\pi[^d$ be such that $\tilde{\gamma}= \gamma+\gamma'\pmod{2\pi}$. We argue that for all $z\in\Z^d$,
\begin{eqnarray}
\tilde{v}_\gamma(x+z,\omega)&=&e^{i (\gamma+\gamma')\cdot z}\tilde{v}_\gamma(x,T_z\omega)
\label{e.N-stat} \\
&=& e^{i \tilde{\gamma}\cdot z}\tilde{v}_\gamma(x,T_z\omega),\nonumber
\end{eqnarray}
so that 
\begin{equation*}
\tilde{v}_\gamma(x)=e^{i \tilde\gamma \cdot x}p(x)
\end{equation*}
for some $p\in \Hh^1_1$. Indeed, consider first $z\in N\mathbb{Z}^d$ in \eqref{e.N-stat}. In this case the $N$-stationarity of $q_N$ and the structure of $\gamma$ yield
\begin{align*}
\tilde{v}_\gamma(x+Ny,\omega)&=\sum_{z\in \{0,\dots,N-1\}^d}e^{-i \tilde{\gamma}\cdot z}e^{i\gamma^{\prime}\cdot Ny}v(x+z,T_{Ny}T_{-z}\omega) 
\\
&= e^{i\tilde{\gamma}\cdot Ny}\sum_{z\in \{0,\dots,N-1\}^d}e^{-i \tilde{\gamma}\cdot z}v(x+z,T_{-z}T_{Ny}\omega) = e^{i\tilde{\gamma}\cdot Ny}\tilde{v}_{\gamma}(x,T_{Ny}\omega)
\end{align*} 
as claimed. It remains to address the case $z\in \{0,\dots,N-1\}^d$  in \eqref{e.N-stat}, and by an iterative argument can be reduced to the case of a unit vector $z=e_i$. We split the sum as $S_1+S_2$ given below. On the one hand,
\begin{align*}
S_1:=&\sum_{\substack{z\in \{0,\dots,N-1\}^d \\ z_i< N-1}}e^{-i \tilde{\gamma}\cdot( z+e_i)}e^{i\tilde{\gamma}\cdot e_i}v(x+z+e_i,T_{-z-e_i}T_{e_i}\omega)
\\
=&\sum_{\substack{z^{\prime}\in \{0,\dots,N-1\}^d \\ z^{\prime}_i>0}}e^{-i \tilde{\gamma}\cdot z^{\prime}}e^{i\tilde{\gamma}\cdot e_i}v(x+z^{\prime},T_{-z^{\prime}}T_{e_i}\omega)
\end{align*}
On the other hand, again by $N$-stationarity and the structure of $\gamma$ we have
\begin{align*}
S_2:=&\sum_{\substack {z\in \{0,\dots,N-1\}^d \\ z_i=N-1}}e^{-i \tilde{\gamma}\cdot z}v(x+z+e_i,T_{-z-e_i}T_{e_i}\omega)
\\
=&\sum_{\substack {z\in \{0,\dots,N-1\}^d \\ z_i=N-1}}e^{-i \tilde{\gamma}\cdot z}e^{i\gamma^{\prime}\cdot Ne_i}v(x+z+(1-N)e_i,T_{-z+(N-1)e_i}T_{e_i}\omega)
\\
=&\sum_{\substack {z^{\prime}\in \{0,\dots,N-1\}^d \\ z^{\prime}_i=0}}e^{-i \tilde{\gamma}\cdot z^{\prime}}e^{i\tilde{\gamma}\cdot e_i}v(x+z^{\prime},T_{-z^{\prime}}T_{e_i}\omega).
\end{align*}
This proves our subclaim since  $\tilde{v}_{\gamma}(x+e_i,\omega)=S_1+S_2 = e^{i\tilde{\gamma}\cdot e_i}\tilde{v}_{\gamma}(x,T_{e_i}\omega)$. Analogously to \eqref{eq:order-3} one establishes
\begin{equation*}
\expec{\int_{Q_N}\bar{\nabla v}\cdot \LL \nabla v}\,=\,
\expec{\int_Q \sum_{\gamma\in I}  \bar{\nabla \tilde v_{\gamma}}\cdot \LL \nabla \tilde v_{\gamma}} .
\end{equation*}
The proof is finished the same way as in Step~2.

\medskip

\noindent \textit{Step~6.} Proof of $\Lambda_1\ge \Lambda_3$.

\noindent Together with the obvious inequality $\Lambda_3\ge \Lambda_2$ (take $\gamma=0$) and the previous steps, this will prove $\Lambda=\Lambda_1=\Lambda_2=\Lambda_3$.

For all $\gamma\in [0,2\pi[^d$ and $p\in \Hh^1_1$ we set $w:x\mapsto e^{i \gamma\cdot x}p (x)$.
We then note that by stationarity of $p$ and $\LL$, and the properties of the Hermitian product, for all $z\in \Z^d$,
\begin{eqnarray}
&N^{-d}\expec{\int_{Q_N}\bar{\nabla w} \cdot \LL \nabla w}\,=\, \expec{\int_Q \bar{\nabla w}\cdot \LL \nabla w}, &\label{eq:order-4} \\
& \expec{\int_{z+Q}|\nabla w|^2}=\expec{\int_{Q}|\nabla w|^2}
,\ \expec{\int_{z+Q}|w|^2}=\expec{\int_{Q}|w|^2}.&\label{eq:order-5} 
\end{eqnarray}
For all $N\geq 2$, we now construct a function $w_N\in \Hh_N^1$ associated with $w$ such that
\begin{equation}\label{eq:order-6}
\expec{\int_{Q_N}\bar{\nabla w_N}\cdot \LL \nabla w_N }\leq N^d\expec{\int_Q \bar{\nabla w} \cdot \LL \nabla w}+CN^{d-1} \expec{\int_Q |\nabla w|^2+|w|^2}.  
\end{equation}
Let $\eta_N$ be a smooth cut-off for the set $(1,N-1)^d$ in $Q_N$ such that $|\nabla \eta_N|\lesssim 1$. 
We then set $w_N(x,\omega):=\sum_{z\in \Z^d}\eta_N(x+Nz)w(x+Nz,T_{-Nz}\omega)$, where the sum is finite almost surely since $\eta_N$ has compact support.
By construction, $w_N\in \Hh^1_N$. Indeed, for all $k\in \Z^d$, by the group property of $T$,
\begin{eqnarray*}
w_N(x+Nk,\omega)&=&\sum_{z\in \Z^d}\eta_N(x+N(k+z))w(x+N(k+z),T_{-Nz}\omega)\\
&=&\sum_{z\in \Z^d}\eta_N(x+N(k+z))w(x+N(k+z),T_{-N(z+k)}T_{Nk}\omega)\\
&=&w_N(x,T_{Nk}\omega),
\end{eqnarray*}
as desired. We now give the argument for \eqref{eq:order-6}. For $x\in Q_N$,
$w_N(x,\omega)=\eta_N(x)w(x,\omega)$, so that
$$
\Big|\int_{Q_N}\bar{\nabla w_N}\cdot \LL \nabla w_N -\int_{(1,N-1)^d}\bar{\nabla w}\cdot \LL \nabla w\Big|\,\leq\,C \int_{Q_N\setminus (1,N-1)^d} |w|^2+|\nabla w|^2,
$$
from which \eqref{eq:order-6} follows by taking the expectation and using  \eqref{eq:order-4} and \eqref{eq:order-5}.
Likewise, 
\begin{equation}\label{eq:order-7}
\expec{\int_{Q_N}|\nabla w_N|^2 }\geq \expec{\int_{(1,N-1)^d} |\nabla w|^2}\,=\,(N-2)^d\expec{\int_{Q} |\nabla w|^2}.
\end{equation}
The combination of \eqref{eq:order-6} and \eqref{eq:order-7} then yields by definition of $\Lambda_3$:
\begin{equation*}
\frac{\expec{\bar{\nabla w}\cdot \LL \nabla w}}{\expec{|\nabla w|^2}} \geq \liminf_{N\to \infty} \frac{\expec{\int_{Q_N}\bar{\nabla w_N}\cdot \LL \nabla w_N }}{\expec{\int_{Q_N}|\nabla w_N|^2} } \geq \Lambda_3.
\end{equation*}
The assertion now follows by taking the infimum over $w$ as in the defining formula  for $\Lambda_1$ .

\medskip

\noindent \textit{Step~7.} Proof of $\Lambda_4\geq \Lambda_5$.

\noindent   Let $a\in \C^d$, $b\in \R^d$, and let $q\in \Hh^1_1$ be such that $|a\otimes b+\nabla q|\not\equiv 0$.
Let $\gamma_n$ be a sequence converging to zero.
We then define the stationary field
\begin{equation*}
p_n(x):=(q(x)+\frac{a}{i \gamma_n}),
\end{equation*}
and set $v_n(x):= e^{i \gamma_nb\cdot x}p_n(x)$.
By definition of $\Lambda_5$,
$$
\lim_{n\uparrow \infty} \frac{\expec{\int_Q\bar{\nabla v_n}\cdot \LL \nabla v_n}}{\expec{\int_Q |\nabla v_n|^2}}\geq \Lambda_5.
$$
On the other hand, 
\begin{equation*}
\nabla v_n(x)=\nabla  \left(e^{i \gamma_nb\cdot x}(q(x)+\frac{a}{i \gamma_n})\right)\,=\, e^{i \gamma_nb\cdot x}\left(\nabla q(x)+(q(x)+\frac{a}{i \gamma_n})\otimes i\gamma_nb\right) ,
\end{equation*}
so that $\lim_{n\uparrow \infty}\expec{\int_Q|\nabla v_n-\nabla q-a\otimes b|^2}=0$, and therefore
$$
\lim_{n\uparrow \infty} \frac{\expec{\int_Q\bar{\nabla v_n} \cdot \LL \nabla v_n}}{\expec{\int_Q |\nabla v_n|^2}}=  \frac{\expec{\int_Q(\bar{a\otimes b+\nabla q}) \cdot \LL (a\otimes b+\nabla q)}}{\expec{\int_Q |a\otimes b+\nabla q|^2}}.
$$
By arbitrariness of $a,b$ and $q$, this implies the claim $\Lambda_4\ge \Lambda_5$.


\subsection{Proof of \tref{alaGMT}}

We split the proof into four steps. We start by regularizing the problem to make the energy functional uniformly coercive, so that it can be homogenized.
Under the assumption that $\Lambda_6>0$ we then show in Steps~2 and 3 that one can pass to the limit in the regularization parameter. We conclude with the discussion of strong ellipticity depending whether $\Lambda_4>0$ or $\Lambda_4=0$.

\medskip 

\noindent \textit{Step~1.} Regularization.

\noindent For all $\eta\ge 0$ set $\LL^\eta=\LL+\eta \mathsf{1}_4$ (where $\mathsf{1}_4$ is the identity 4-tensor), and for all $\e>0$ set $\LL^\eta_\e:=\LL^\eta(\frac{\cdot}{\e})$.
Since $\Lambda\ge 0$, for all $\eta,\e>0$, Korn's inequality implies that the functional
$$
u\in \K \mapsto \F_\e^\eta(u):=\int_D \nabla u\cdot \LL^\eta_\e \nabla u
$$
satisfies for all $u\in \K$
$$
\F_\e^\eta(u)\,\gtrsim\, \eta \int_D |\nabla u|^2.
$$
It is then standard to prove that for $\eta>0$, $\F_\e^\eta$ $\Gamma(w-H^1)$-converges almost surely on $\K$ to the integral functional $\F_*^\eta:\K\to \R_+$ defined  by
$$
\F_*^\eta(u)\,=\,\int_D \nabla u \cdot \LL_*^\eta\nabla u
$$
for a strictly strongly elliptic matrix $\LL_*^\eta$ given in direction $M\in \Md$ by
$$
M\cdot \LL^\eta_* M\,=\,\inf_{\phi \in \Hhr^1_1} \expec{\int_Q(M+\nabla \phi)\cdot \LL^\eta(M+\nabla \phi)}.
$$
In particular, for all $\eta\ge 0,\e>0$, $\LL^\eta_\e\ge \LL_\e$ (in the sense of quadratic forms), and for all $u\in H^1_0(D)$, $\F^\eta_\e(u)\ge \F_\e(u)$.
Likewise, $\LL^\eta_*\ge \LL_*$ and for all $u\in \K$,  $\F^\eta_*(u)\ge \F_*(u)\,:=\,\int_D\nabla u \cdot\LL_*\nabla u$.

\medskip

\noindent \textit{Step~2.} $\Gamma$-liminf inequality.

\noindent Let us prove that for all $u\in \K$ and all $u_\e\in \K$ such that $u_\e \rightharpoonup u$ weakly in $H^1(D)$ we have almost surely
\begin{equation}\label{eq:GMT-1}
\liminf_{\e\downarrow 0} \F_\e(u_\e)\,\geq \, \F_*(u).
\end{equation}
On the one hand, for all $\eta>0$, the $\Gamma$-convergence of $\F_\e^\eta$ to $\F^\eta_*$ yields
\begin{equation}\label{eq:GMT-2}
\liminf_{\e\downarrow 0} \F_\e^\eta(u_\e)\,\geq \, \F_*^\eta(u).
\end{equation}
On the other hand, since $u_\e$ converges to $u$ weakly in $H^1(D)$, $u_\e$ is bounded by some finite constant $\sqrt C$ in $H^1(D)$ and
\begin{equation}\label{eq:GMT-3}
\F^\eta_\e(u_\e) \le \F_\e(u_\e)+\eta\|\nabla u_\e\|_{L^2(D)}^2 \,\leq \, \F_\e(u_\e)+C\eta.
\end{equation}
Hence, by \eqref{eq:GMT-2},
\begin{equation*}
\F_*(u)\le \F^\eta_*(u) \le  \liminf_{\e\downarrow 0}\F^{\eta}_{\e}(u_{\e})\le \liminf_{\e\downarrow 0}\F_{\e}(u_{\e})+\,C \eta.
\end{equation*}
This proves \eqref{eq:GMT-1} by the arbitrariness of $\eta>0$.

\medskip

\noindent \textit{Step~3.} Construction of a recovery sequence.
\nopagebreak
\noindent
Let $M\in \Md$ be fixed. We start by showing that a corrector exists.
Since $\Lambda_6\ge\Lambda\ge 0$, the map 
$$\L_M:\Hhr^1_1\ni q \mapsto \expec{\int_Q (M+\nabla q) \cdot \LL (M+\nabla q)}$$ 
is convex, and therefore weakly lower-semicontinuous for the weak convergence of gradient fields.
In addition, since $\LL$ is uniformly bounded, $\Lambda_6>0$ implies by Cauchy-Schwarz' and Young's inequalities (to control the two linear terms in $\nabla q$) that there exists $C_M>0$ such that for all $q\in \Hhr^1_1$, $\L_M$ satisfies
$$
\L_M(q)\,\ge\,\frac{\Lambda_6}{2} \expec{\int_Q|\nabla q|^2}-\frac{C_M}{\Lambda_6}.
$$
It is then standard to show there exists a unique potential field $\Phi_M \in \Hhr^0_1$ such that 
$$
\inf_{q\in \Hhr^1_1} \L_M(q)\,=\,\expec{\int_Q (M+\Phi_M)\cdot \LL (M+\Phi_M)}.
$$
We then denote by $\phi_M$ the unique random field such that $\nabla \phi_M=\Phi_M$ almost everywhere almost surely that satisfies $\int_B \phi_M=0$ almost surely (note that $\phi_M$ is not stationary, whereas $\nabla \phi_M$ is). 
Since $\nabla \phi_M$ is stationary,  $\phi_M$ is sublinear at infinity in the sense that
almost surely
\begin{equation}\label{eq:GMT-3.1}
\lim_{R\uparrow \infty} \fint_{B_R} R^{-2}|\phi_M|^2 \,=\,0.
\end{equation}
We further define a 3-tensor $\pi$ and a 4-tensor $\Pi$ such that 
for all $M\in \Md$,
$$
\pi(x)M\,=\,\phi_M(x),\quad \text{and }\Pi(x)M\,=\,\nabla \phi_M(x)
$$
(which is possible since $\phi_M$ depends linearly on $M$).
Let now $u\in \K$. 
We are in position to construct the desired recovery sequence, that is, for almost every realization, a sequence $u_\e \in \K$ that weakly converges to $u$ in $H^1(D)$ and that saturates the $\Gamma-\liminf$ inequality
in the sense that
$$
\lim_{\e \downarrow 0} \F_\e(u_\e)\,=\,\F_*(u).
$$
As well-known to the expert in homogenization, it is enough to prove the result for $u\in C^2(\bar D)$ and for a sequence  $u_\e \in H^1(D)$ (and not necessarily $u_\e \in \K$).
Indeed, the extension to $u\in \K$ with $u_\e \in \K$ is then proved by approximation and cut-off using Attouch' diagonalization lemma (randomness or periodicity plays no role provided we have \eqref{eq:GMT-3.2} below, see for instance \cite[Section~4.5, Step~2]{GMT-93}).
We thus take $u\in C^2(\bar D)$ and define $u_\e \in H^1(D)$ by
$$
u_\e(x)\,:=\,u(x)+\e \pi(\frac{x}{\e})\nabla u(x).
$$
In view of \eqref{eq:GMT-3.1}, $u_\e$ converges strongly to $u$ in $L^2(D)$ almost surely.
Likewise, since
\begin{equation}\label{eq:GMT-3.3}
\nabla u_\e(x)\,=\,\nabla u(x)+\Pi(\frac{x}{\e})\nabla u(x)+\e \pi(\frac{x}{\e})\nabla^2 u(x), 
\end{equation}
we have
$$
\int_D |\nabla u_\e|^2 \,\lesssim\, \int_D |\nabla u|^2 +\int_D |\Pi(\frac{x}{\e})|^2 |\nabla u(x)|^2+o(1).
$$
Combined with the ergodic theorem in the form of 
\begin{equation*}
\lim_{\e\downarrow 0}\int_D|\Pi(\frac{x}{\e})|^2 |\nabla u(x)|^2dx\,=\,\expec{\int_Q |\Pi|^2}\int_D |\nabla u|^2,
\end{equation*}
this turns into
\begin{equation}\label{eq:GMT-3.2}
\limsup_{\e \downarrow 0} \int_D |\nabla u_\e|^2 \,\lesssim \, \int_D |\nabla u|^2,
\end{equation}
which is indeed needed for the approximation argument (see \cite[Section~4.5, Step~2]{GMT-93}).
Since the last term in the right-hand side of \eqref{eq:GMT-3.3} converges strongly to zero in $L^2(D)$, we have almost surely
$$
\lim_{\e \downarrow 0} |\int_D\nabla u_\e \cdot \LL_\e \nabla u_\e
-\int_D \nabla u(x)\cdot (\mathsf{1}_4+\Pi(\frac{x}{\e})) \LL_\e(x) 
(\mathsf{1}_4+\Pi(\frac{x}{\e}))\nabla u(x)dx|\,=\,0.
$$
A last use of the ergodic theorem in the form of
\begin{eqnarray*}
\lefteqn{\lim_{\e \downarrow 0} \int_D \nabla u(x)\cdot (\mathsf{1}_4+\Pi(\frac{x}{\e})) \LL_\e(x) 
(\mathsf{1}_4+\Pi(\frac{x}{\e}))\nabla u(x)dx}
\\
&=& \int_D \nabla u(x)\cdot \expec{\int_Q (\mathsf{1}_4+\Pi(y)) \LL(y) 
(\mathsf{1}_4+\Pi(y))dy}\nabla u(x)dx
\\
&=& \int_D \nabla u \cdot \LL_* \nabla u
\end{eqnarray*}
yields the desired $\Gamma-\limsup$-inequality
$$
\lim_{\e \downarrow 0}\F_\e(u_\e)\,=\,\F_*(u).
$$

\medskip

\noindent \textit{Step~4.} Strong ellipticity of $\LL_*$.

\noindent We first treat the case when $\Lambda_4=0$.
Then there exist (real-valued) $a_n\otimes b_n$ and $q_n\in \Hh^1_1$ such that 
$\expec{\int_Q |a_n\otimes b_n+\nabla q_n|^2}=|a_n\otimes b_n|^2+\expec{\int_Q|\nabla q_n|^2}=1$ and
$$
\lim_{n\uparrow \infty} \expec{\int_Q (a_n\otimes b_n+\nabla q_n)\cdot \LL(a_n\otimes b_n+\nabla q_n)}\,=\,0.
$$
Since $|a_n\otimes b_n|\le 1$, we may assume (up to taking a subsequence) that $a_n\otimes b_n \to a\otimes b$. Using in addition that $\nabla q_n$ is bounded, this yields (along the subsequence)
\begin{equation}\label{eq:GMT-4.1}
\lim_{n\uparrow \infty} \expec{\int_Q (a\otimes b+\nabla q_n)\cdot \LL(a\otimes b+\nabla q_n)}\,=\,0.
\end{equation}
By definition of $\LL_*$, this implies $a\otimes b\cdot \LL_*a\otimes b=0$, so that $\LL_*$ loses ellipticity at $a\otimes b$ provided  $a\otimes b\neq 0$. Assume momentarily that $a\otimes b=0$.
Then for $n$ large enough, $\expec{\int_Q |\nabla q_n|^2}\ge \frac{1}{2}$ and \eqref{eq:GMT-4.1} turns into
$$
\lim_{n\uparrow \infty} \frac{\expec{\int_Q\nabla q_n\cdot \LL\nabla q_n}}{\expec{\int_Q |\nabla q_n|^2}}\,=\,0,
$$
which contradicts $\Lambda_6>0$.

\medskip

Assume now that $\Lambda_4>0$ and fix $a\otimes b \neq 0$.
Let $q\in \Hhr^1_1$ be a test-function for the minimizing problem defining $\LL_*$ in direction $a\otimes b$.
Then $\expec{\int_Q|a\otimes b+\nabla q|^2}=|a\otimes b|^2+\expec{\int_Q |\nabla q|^2}\ge |a\otimes b|^2$, so that
$$
\frac{\expec{\int_Q(a\otimes b+\nabla q)\cdot \LL(a\otimes b+\nabla q)}}{|a\otimes b|^2}\,\ge \, \frac{\expec{\int_Q(a\otimes b+\nabla q)\cdot \LL(a\otimes b+\nabla q)}}{\expec{\int_Q |a\otimes b+\nabla q|^2}}\,\geq \,\Lambda_4,
$$
which proves that $\Lambda_*\ge \Lambda_4$.


\section{Proofs: Estimate of the ellipticity constants}\label{sec:proofs-ell}

\subsection{Proof of \pref{LAMBDA}: $\Lambda\ge 0$}

We shall prove a lower bound on the energy, which is more precise than stated in \pref{LAMBDA} and which will be used later on.
We decompose $\LL$ as $\LL-\underline \LL+\underline \LL$ where $\underline \LL$ is the isotropic stiffness tensor of Lam\'e constants $\underline \lambda,\underline \mu$
defined as follows: $\underline\mu:=\mu_1$ and $\underline \lambda:=-2\mu_1$. Then, on the one hand, for any matrix $A\in\R^{2\times 2}$ (the entries of which
are denoted by $a_{ij}$) we have (cf.~\eqref{a+5})
\begin{equation*}
A\cdot\underline{\LL}A=-4\mu_1 a_{11}a_{22}+\mu_1 (a_{12}+a_{21})^2=-4\mu_1\det(A)+\mu_1(a_{12}-a_{21})^2.
\end{equation*}
On the other hand, on the set $\{x\,|\,\LL(x)=\LL^1\}$ we have
\begin{equation*}
A\cdot(\LL^1-\underline{\LL})A\, = \chi (\lambda_1+2\mu_1)(a_{11}+a_{22})^2,
\end{equation*}
whereas on the complementary set there holds
\begin{align*}
A\cdot(\LL^2-\underline{\LL})A=&(\lambda_2+2\mu_2)(a_{11}^2+a_{22}^2)+2(\lambda_2+2\mu_1)a_{11}a_{22}+(\mu_2-\mu_1)(a_{12}+a_{21})^2
\\
=&(\mu_2-\mu_1)(a_{11}^2+a_{22}^2)-2(\mu_2-\mu_1)a_{11}a_{22}+(\mu_2-\mu_1)(a_{12}+a_{21})^2
\\
&+(\lambda_2+\mu_2+\mu_1)(a_{11}^2+a_{22}^2)+2(\lambda_2+\mu_2+\mu_1)a_{11}a_{22}
\\
\geq&(\mu_2-\mu_1)\big((a_{11}-a_{22})^2+(a_{12}+a_{21})^2\big),
\end{align*}
where we used in the last estimate that $\lambda_2+\mu_2+\mu_1\geq 0$. Combining the above estimates, we deduce from \eqref{eq:cond} that for 
$\alpha=\min\{\mu_1,\mu_2-\mu_1\}>0$ we have
\begin{align}\label{e.lowerbound}
A\cdot\LL A\geq& -4\mu_1\det(A)+ \chi \alpha \big((a_{11}+a_{22})^2 + (a_{12}-a_{21})^2\big)\nonumber
\\
&+(1-\chi)\alpha\big((a_{11}-a_{22})^2+(a_{12}^2+a_{21}^2)\big).
\end{align}
Now given any $u\in C_0^{\infty}(\R^2,\R^2)$, we observe that \eqref{e.lowerbound} yields
\begin{equation*}
\int_{\R^2}\nabla u\cdot \LL \nabla u\geq -4\mu_1\int_{\R^2}\det(\nabla u) = 0,
\end{equation*}
from which the claim $\Lambda\ge 0$ follows.


\subsection{Properties of sequences of equi-bounded energy}

We first establish a general property of sequences of equi-bounded energy  that will be used several times in the  proofs. 
\begin{lemma}\label{l.limiteq}
Let $\LL:\R^2\times \Omega \to \{\LL^1,\LL^2\}$ be a stationary random field satisfying \eqref{eq:cond}. Let $M\in\R^{2\times 2}$ be such that $\det(M)=0$ and assume 
there exists a sequence $v_n\in\Hhr^1_1$ that satisfies $\sup_n\expec{\int_Q|\nabla v_n|^2}< \infty$ and
\begin{equation}\label{a+3}
\lim_{n\to \infty}\expec{\int_Q(M+\nabla v_n)\cdot\LL(M+\nabla v_n)}=0.
\end{equation}	
Then, up to a subsequence, $\nabla v_n$ converges weakly in $L^2(\Omega\times Q)$ to a stationary field $U:\R^2\times \Omega\to \R^{2\times 2}$ with $U\in L^2(\Omega\times Q)$ satisfying the following property: for $\mathbb{P}$-almost every $\omega$ there exists a function $v=v_{\omega}\in H^1_\loc(\R^2,\R^2)$ such that $U=\nabla v$. 
In addition, with $l_M$ the linear function $y\mapsto My$, there exists almost surely a potential $\psi^M\in H^2_\loc(\R^2)$ such that $l_M+v=\nabla\psi^M$. On the set $\{x\,|\,\LL(x)=\LL^1\}$, it is harmonic, i.e.
\begin{equation*}
-\triangle\psi^M\,=\,0
\end{equation*}
whereas on each connected component $I_j$ of the set $\{x\,|\,\LL(x)\neq \LL^1\}$ there exist real parameters $a_j,b_j,c_j,d_j\in\R$ such that
\begin{equation*}
\psi^M(y)\,=\,a_j(y_1^2+y_2^2)+b_jy_1+c_jy_2+d_j\quad\quad\text{on }I_j.
\end{equation*}
\qed
\end{lemma}
\begin{proof}[Proof of \lref{limiteq}]
We split the proof into two steps.

\medskip

\noindent \textit{Step~1.} Construction of a potential field $U$.

\noindent Since $\sup_n\expec{\int_Q |\nabla v_n|^2}<\infty$, there exists a stationary field $U:\R^d\times \Omega\to \R^{2\times 2}$  bounded in $L^2(\Omega\times Q)$ 
and such that (along a subsequence, not relabelled)
$\nabla v_n\rightharpoonup U$ weakly in $L^2(\Omega\times Q)$. 
Let us give the standard argument that proves that $U$ is almost surely a potential field on $\R^2$.
It is enough to show that for all random variables $\theta \in L^2(\Omega)$ and all $\eta \in C^\infty_c(\R^2,\R^2)$
we have for $i=1,2$,
\begin{equation}\label{a+1}
\expec{\theta \int_{\R^2} (U^i_2 \partial_1 \eta -U_1^i \partial_2 \eta )}\,=\,0.
\end{equation}
Set $\xi:=\sum_{z\in \Z^d} \theta(T_{-z}\omega) \eta(z+\cdot)$, which is well-defined since $\eta$ has compact support.
By definition, $\xi$ is stationary and it belongs to $\Hhr^1_1$.
Since $\nabla v_n$ is a gradient,  the Schwarz' commutation rule yields
\begin{eqnarray*}
0&=&\expec{\theta \int_{\R^2} (\partial_2 v_n^i  \partial_1 \eta -\partial_1 v_n^i  \partial_2 \eta )}
\end{eqnarray*}
Combined with the definition of $\xi$ and the fact that $T_z$ preserves the measure, this takes the form
\begin{eqnarray*}
0&=&\expec{ \int_{Q} (\partial_2 v_n^i \partial_1 \xi -\partial_1 v_n^i \partial_2 \xi )}.
\end{eqnarray*}
Passing to the limit $n\uparrow \infty$ in this identity, and using again the definition of $\xi$, \eqref{a+1} follows.
Since $U$ is a potential field, for all balls $B_j\subset\R^2$ there exists $v_{j,\omega}\in H^1(B_j,\R^2)$ such that $U(\omega,\cdot)=\nabla v_{j,\omega}$ on $B_j$
(see e.g.~\cite[Lemma 3]{Bourgain2000}). By a diagonal construction, we obtain a function $v_{\omega}\in H^1_\loc(\R^2,\R^2)$ such that  $U(\omega,\cdot)=\nabla v_{\omega}$ in $\R^2$ on a set of full probability. It remains to prove the existence and the equations for the potential $\psi^M$.
	
\medskip

\noindent \textit{Step~2.} Construction and properties of $\psi^M$.	
	
\noindent  First note that the null Lagrangian $\det \nabla u$ satisfies $\expec{\int_Q \det \nabla u}\equiv 0$ for all $u\in\Hhr^1_1$,
which indeed follows from a direct integration by parts for $d=2$:
\begin{equation}\label{a+4}
\expec{\int_Q \det \nabla u}\,=\,\expec{\int_Q \partial_1 u^1\partial_2 u^2-\partial_2 u^1 \partial_1 u^2}\,=\, \expec{\int_Q \partial_1 u^1\partial_2 u^2-\partial_1 u^1 \partial_2 u^2}\,=\,0.
\end{equation}
Likewise, if $\det(M)=0$, the same argument combined with a linear expansion of the determinant with respect to the column vectors
yields or all $u\in\Hhr^1_1$
\begin{equation}\label{a+2}
\expec{\int_Q\det(M+\nabla u)}=0.
\end{equation}
The combination of~\eqref{a+2} and~\eqref{e.lowerbound} then yields (recall that $\chi$ is the 
indicator function of the set on which $\LL$ takes value $\LL^1$)
\begin{eqnarray*}
\lefteqn{\expec{\int_Q (M+\nabla v_n)\cdot \LL(M+\nabla u_n)}}
\\
&\ge &
\alpha \expec{\int_Q \chi  \big(M_{11}+\partial_1 v^1_n+M_{22}+\partial_2 v^2_n\big)^2}
\\
&&+\alpha  \expec{\int_Q \chi\big(M_{12}+ \partial_2 v^1_n-M_{21}-\partial_1 v^2_n\big)^2}
\\
&&+\alpha \expec{\int_Q (1-\chi)\big(M_{11}+\partial_1 v^1_n-M_{22}-\partial_2 v^2_n)^2}
\\
&&+\alpha\expec{\int_Q (1-\chi)\big((M_{12}+\partial_2 v^1_n)^2+(M_{21}+\partial_1 v^2_n)^2\big)}.
\end{eqnarray*}
Since each of the RHS terms is a non-negative quadratic form, assumption~\ref{a+3} implies by weak convergence
of $\nabla v_n$ to $\nabla v$ in $L^2(\Omega\times Q)$:
On the set $\{x\,|\,\LL(x)=\LL^1\}$,
\begin{align}
M_{11}+\partial_1 v^1+M_{22}+\partial_2 v^2&=0,\label{eq:PDE1}\\
M_{12}+\partial_2 v^1&=M_{21}+\partial_1 v^2,\label{eq:PDE2}
\end{align} 
while on the complementary set $\{x\,|\,\LL(x)\neq \LL^1\}$, we have
\begin{align}
M_{11}+\partial_1 v^1&=M_{22}+\partial_2 v^2,\label{eq:PDE3}\\
\big(M_{12}+\partial_2 v^1\big)^2&=\big(M_{21}+\partial_1 v^2 \big)^2=0.\label{eq:PDE4}
\end{align}
The combination of \eqref{eq:PDE2} and \eqref{eq:PDE4} implies the existence of a potential $\psi^M\in H^2_{\rm loc}(\R^2)$ such that $\nabla\psi^M=l_M+ v$. On the set $\{x\,|\,\LL=\LL^1\}$ equation \eqref{eq:PDE1} yields that $\triangle\psi^M=0$, while on each connected component of the set $\{x\,|\,\LL\neq\LL^1\}$ equation \eqref{eq:PDE4} shows that $\partial_i \psi^M$ is a function of $y_i$ only, so that \eqref{eq:PDE3} implies that both functions are linear with the same derivative. In particular, for each connected component $I_j$ of $\{x\,|\,\LL(x)\neq\LL^1\}$ there exist parameters $a_j,b_j,c_j,d_j\in\R$ such that $\psi^M(y)=a_j(y_1^2+y_2^2)+b_jy_1+c_jy_2+d_j$ on that component.
\end{proof}

In the next lemma, we strengthen the conclusion of \lref{limiteq} by exploiting Assumptions~\ref{a.non-neg}, ~\ref{b.non-neg} or~\ref{c.non-neg}. This is the 
only place where the precise geometric structure of the boundary is used.
\begin{lemma}\label{l.weaktozero}
Let the random field $\LL:\R^2\times \Omega$ satisfy Assumption~\ref{a.non-neg},~\ref{b.non-neg} or~\ref{c.non-neg}, and let $M\in\R^{2\times 2}$ be such that $\det(M)=0$. 
If there exists a sequence $v_n\in\Hhr^1_1$ satisfying $\sup_n\expec{\int_Q|\nabla v_n|^2}<\infty$ and
\begin{equation*}
\lim_{n\to \infty}\int_Q(M+\nabla v_n)\cdot\LL(M+\nabla v_n) = 0,
\end{equation*}
then $M=0$ and $\nabla v_n\rightharpoonup 0$ in $L^2(\Omega\times Q)$.

If $\LL$ satisfies Assumption~\ref{c.non-neg} and (C3-b), we rather have
\begin{itemize}
\item if $\theta_1\ne \frac12$, then $M=0$ and  $\nabla v_n\rightharpoonup 0$ in $L^2(\Omega\times Q)$;
\item if $\theta_1=\theta_2=\frac12$, then there exists $\kappa\in \R$ such that $M=\kappa e_2\otimes e_2$,
and if $\kappa=0$ then $\nabla v_n\rightharpoonup 0$ in $L^2(\Omega\times Q)$.
\end{itemize}
\qed
\end{lemma}
\begin{proof}[Proof of \lref{weaktozero}]
We split the proof into four steps.
In the first step, we apply \lref{limiteq} to the sequence $v_n$. 
In the second, third, and fourth steps, we prove the claim under Assumption~\ref{a.non-neg}, Assumption~\ref{b.non-neg}, and Assumption~\ref{c.non-neg}, respectively.

\medskip

\noindent \textit{Step~1.} Weak limit of $\nabla v_n$.

\noindent 
By \lref{limiteq}, up to taking a subsequence, $\nabla v_n\rightharpoonup \nabla v$ for some $\nabla v\in L^2(\Omega \times Q)$, which agrees indeed with a gradient for almost every $\omega\in\Omega$.
In addition, there exists almost surely a potential $\psi\in H^2_\loc(\R^2)$ such that $l_M+v=\nabla \psi$, and 
such that 
\begin{itemize}
\item for each connected component $I_j$ of $\{x\,|\,\LL(x)\neq\LL^1\}$ we have $\psi(y)=a_j(y_1^2+y_2^2)+b_jy_1+c_jy_2+d_j$ for some parameters $a_j,b_j,c_j,d_j\in\R$, 
\item on the set $\{x\,|\,\LL(x)=\LL^1\}$, $\triangle\psi = 0$.
\end{itemize}
Since $\psi\in H^2_{\rm loc}(\R^2)$ and all the connected components $I_j$ have Lipschitz-boundary by assumption,   the traces of $\psi$ and $\nabla\psi$ must agree on both sides of $I_j$. Thus $\psi$ solves the following (hopefully) overdetermined elliptic equation:
\begin{equation}\label{eq:THEPDE}
\begin{cases}
\triangle \psi = 0 &\mbox{on $\R^2\setminus\bigcup_j \overline{I_j}$,} \\
\psi = a_j(y_1^2+y_2^2)+b_jy_1+c_jy_2+d_j&\mbox{in $I_j$,}\\
\nabla\psi = \begin{pmatrix*}
2 a_jy_1+b_j\\
2a_jy_2+c_j
\end{pmatrix*} &\mbox{on $\partial I_j$.}
\end{cases}
\end{equation}

\medskip

\noindent \textit{Step~2.} Proof under Assumption~\ref{a.non-neg}.

\noindent In this case, \eqref{eq:THEPDE} implies that $\psi(y)=a(y_1^2+y_2^2)+b_1y_1+b_2 y_2+d$ on the connected set $\{x\,|\,\LL\neq\LL^1\}$
for some parameters $a$, $b_1$, $b_2$ and $d$, and $\psi$ is harmonic on $\{x\,|\,\LL=\LL^1\}$.
Since $\psi \in H^2_\loc(\R^d)$, for $i=1,2$, $\partial_i \psi(y)=2 ay_i+b_i$ is linear on  $\{x\,|\,\LL\neq\LL^1\}$, harmonic on the inclusions $I_j \subset \{x\,|\,\LL=\LL^1\}$,
and have boundary value $y\mapsto 2 ay_i+b_i$ on $\partial I_j$. By uniqueness of the Dirichlet problem on the inclusions $I_j$, 
$\nabla \psi$ is linear on $\R^2$ and given by
$$
\nabla \psi(y)= 
\begin{pmatrix*}
2ay_1+b_1\\
2ay_2+b_2
\end{pmatrix*} .
$$
Hence, $\nabla^2 \psi=2a \Id$.
We then appeal to the identity $\nabla^2 \psi =M+\nabla v$ and to the ergodic theorem in the form $\lim_{R\uparrow \infty} \fint_{Q_R} \nabla v=\expec{\int_Q \nabla v}=0$,
to the effect that $2a \Id=M$. Since $\det (M)=0$, this implies $a=0$ and $M=0$, and therefore also $\nabla v\equiv 0$.

\medskip

\noindent \textit{Step~3.} Proof under Assumption~\ref{b.non-neg}.

\noindent We argue that if Assumption~\ref{b.non-neg} holds, \eqref{eq:THEPDE} is possible only if $a_j=0$ for all $j$. 
Without loss of generality we assume that $\partial I_1$ contains some straight segment $S$ or an arc of circle $C$. 
Next, given $\psi$ solving \eqref{eq:THEPDE} and a rigid motion $R$, the function $\psi_{R}=\psi\circ R^{-1}$ solves the system
\begin{equation*}
\begin{cases}
\triangle \psi_{R} = 0 &\mbox{on $\R^2\setminus\bigcup_i R(\overline{I_j})$,} \\
\psi_R=a_j(y_1^2+y_2^2)+\tilde{b_j}y_1+\tilde{c_j}y_2+\tilde{d_j} &\mbox{in $R(I_j)$,}\\
\nabla\psi_R = \begin{pmatrix*}
2 a_jy_1+\tilde{b_j}\\
2a_jy_2+\tilde{c_j}
\end{pmatrix*} &\mbox{on $\partial R(I_j)$.}
\end{cases}
\end{equation*}
Hence, without loss of generality, up to replacing $\psi$ by $\psi_R$, we can assume that $S=[0,t]\times\{0\}$ or $C\subset \partial B_t$ for some $t>0$. 
Consider then the two functions
\begin{align*}
\psi_{S}(y)&= a_1(y_1^2-y_2^2)+b_1 y_1 +c_1y_2+d_1,\\
\psi_{C}(y)&=a_1\log(y_1^2+y_2^2)-a_1\log(t^2)+a_1t^2+b_1 y_1 +c_1y_2+d_1,
\end{align*}
which satisfy $\triangle\psi_{S}=0$ on $\R^2$ and $\triangle\psi_C=0$ on $\R^2\setminus\{0\}$, respectively. 
They also satisfy the boundary conditions $\psi_{S}=\psi$ and $\nabla\psi_{S}=\nabla\psi$ on $S$ or $\psi_C=\psi$ and $\nabla\psi_C=\nabla\psi$ on $C$, respectively. 
By Holmgren's uniqueness theorem, we deduce that  $\psi=\psi_{S}$ or $\psi=\psi_C$ on an open set and therefore on the connected set $\R^2\setminus\bigcup_i\overline{I_j}$. 
Consider now an arbitrary component $I_j$ such that $0\notin I_j$. By \eqref{eq:THEPDE}, the functions $\nabla\psi_{S}$ and $\nabla\psi_C$ are linear (and thus harmonic) on $\partial I_j$. By uniqueness for the Dirichlet problem on $I_j$ we obtain the following compatibility conditions 
\begin{align*}
\begin{pmatrix*}
2a_jy_1+b_j\\
2a_jy_2+c_j
\end{pmatrix*}=\begin{pmatrix*}
2a_1y_1+b_1\\
-2a_1y_2+c_1
\end{pmatrix*} \quad\text{or}\quad\begin{pmatrix*}
2a_jy_1+b_j\\
2a_jy_2+c_j
\end{pmatrix*}=\begin{pmatrix*}
2a_1\frac{y_1}{y_1^2+y_2^2}+b_1\\
2a_1\frac{y_2}{y_1^2+y_2^2}+c_1
\end{pmatrix*}\quad\mbox{on $I_j$.}
\end{align*}
Since $I_j$ is an open set, this yields $a_j=a_1=0$, and $b_j=b_1$ and $c_j=c_1$. Thus the restriction that $0\notin I_j$ was not necessary and we conclude that there exist $b,c\in \R$ such that
 $\nabla \psi=\begin{pmatrix*}
b\\
c
\end{pmatrix*}=l_M+v$. 
Taking the gradient of this identity and using the ergodic theorem in the form $\lim_{R\uparrow \infty} \fint_{Q_R} \nabla v=\expec{\int_Q \nabla v}=0$,
we deduce that $M=0$, and therefore also $\nabla v=\nabla^2\psi-M=0$, as claimed.

\medskip

\noindent \textit{Step~4.} Proof under Assumption~\ref{c.non-neg}.

\noindent Again we exploit the PDE \eqref{eq:THEPDE}. We claim that the coefficients $a_j$ are independent of $j$. Here comes the argument: Consider the connected component $K_1$ of the set $\{x\,|\,\LL(x)=\LL^1\}$ which, we assume, contains a segment $S=R^{-1}_1(\{0\}\times [0,t])\subset\partial K_1$, where $R_1(y)=Q_1(y)+q_l$ is a rigid motion and $t>0$. Then by Lipschitz regularity, we can find a connected component $I_{j_1}$ of $\{x\,|\,\LL(x)\neq\LL^1\}$ such $S\subset\partial I_{j_1}$. As in {\it Step 3}, Holmgren's uniqueness theorem implies that on $K_1$ we have
\begin{equation*}
\psi(y) = y^TQ_1^T \begin{pmatrix} -a_{j_1} &0 \\ 0 &a_{j_1}\end{pmatrix}Q_1y+\bar{b}_1y_1+\bar{c}_1y_2+\bar{d}_1
\end{equation*}
for some constants $\bar{b}_1,\bar{c}_1,\bar{d}_1\in\R$. Our focus lies on the gradient $\nabla\psi$ which on the set $K_1$ takes the form
\begin{equation}\label{e.formulagradient}
\nabla\psi (y) = Q_1^T \begin{pmatrix} -2a_{j_1} &0 \\ 0 &2a_{j_1}\end{pmatrix}Q_1y+\begin{pmatrix}
 \bar{b}_1 \\ \bar{c}_1
\end{pmatrix}.
\end{equation}

We first treat Assumption~(C3-a) and show that $a_{j_1}=0$. To this end, we derive a compatibility condition at a non-flat part of the boundary $\partial K_1$ (we assume that $\partial K_1$ is not straight). By assumption there exists a Lipschitz curve $\gamma_*:(-\delta,\delta)\to\R^2$ and a connected component of $I_n$ of the set $\{x\,|\,\LL(x)\neq\LL^1\}$ such that $\gamma_*((-\delta,\delta))\subset\partial K_1\cap \partial I_n$  is non-flat. The boundary values of $\nabla\psi$ on $\partial I_{n}$ and \eref{formulagradient} then yield
\begin{equation*}
Q_1^T \begin{pmatrix} -2a_{j_1} &0 \\ 0 &2a_{j_1}\end{pmatrix}Q_1\gamma_*(t)+\begin{pmatrix}
\bar{b}_1 \\ \bar{c}_1
\end{pmatrix} = 2a_{n}\gamma_*(t)+\begin{pmatrix}
b_{n} \\ c_{n}
\end{pmatrix}
\end{equation*}
for every $t\in(-\delta,\delta)$ (note that $\gamma_*$ preserves $\mathcal{H}^1$-null sets and the above expression is continuous in $t$). Since $\gamma_*$ is Lipschitz, we can differentiate the above identity for almost every $t$, which yields after multiplication with $Q_1$
\begin{equation}\label{e.compcomp}
\begin{pmatrix} -a_{j_1}-a_n &0 \\ 0 &a_{j_1}-a_n\end{pmatrix}Q_1\gamma^{\prime}_*(t)=0.
\end{equation}
Non-flatness implies that there exist two linearly independent vectors $v_1,v_2\in\R^2$ and points $t_1,t_2$ satisfying \eref{compcomp} such that $\gamma_*^{\prime}(t_i)=v_i$ for $i=1,2$. Testing those $t_i$ yields $a_{j_1}=a_n$ and $a_{j_1}=-a_n$. Consequently we have $a_{j_1}=0$.

Now fix any connected component $I_{j_0}$ of $\{x\,|\,\LL(x)\neq\LL^1\}$ with $j_0\neq j_1$. First we construct a finite sequence of connected components $(I_{j_m})_m$ and $(K_{n_m})_m$ which join $I_{j_1}$ and $I_{j_0}$ in a suitable sense. Fix a point $p_{j_0}$ in the interior of $I_{j_0}$ and choose $x_1\in\R^2$ by requiring
\begin{equation*}
x_1 \in {\rm argmin}\{|x-p_{j_0}|:\,x\in \overline{I_{j_1}}\}.
\end{equation*}
Note that $x_1\neq p_{j_0}$ and therefore $x_1\in\partial I_{j_1}$. By Lipschitz regularity there exists a unique connected component $K_{n_1}$ of the set $\{x\,|\,\LL(x)=\LL^1\}$ such that $x_1\in\partial K_{n_1}$. Then define $z_1\in\R^2$ by
\begin{equation*}
z_1\in{\rm argmin}\{|z-p_{j_0}|:\,z\in\overline{K_{n_1}}\}.
\end{equation*}
Again $z_1\neq p_{j_0}$ and thus $z_1\in\partial K_{n_1}$ and there exists a unique connected component $I_{j_2}$ of the set $\{x\,|\,\LL(x)\neq\LL^1\}$ such that $z_1\in\partial I_{j_2}$. Continuing this procedure we claim that $x_{m_0}=p_{j_0}$ for some $m_0\in\mathbb{N}$ (in which case we stop the algorithm). Indeed, first note that by construction for all $m\geq 2$ we have
\begin{equation*}
|z_{m}-p_{j_0}|\leq|x_m-p_{j_0}|\leq |z_{m-1}-p_{j_0}|\leq |x_{m-1}-p_{j_0}|.
\end{equation*}
In particular $(x_m,z_m)$ is a bounded sequence. Since by Lipschitz regularity the number of connected components is locally finite, both sequences $|x_m-p_{j_0}|$ and $|z_m-p_{j_0}|$ are finally constant, that is, there exists $m_0\in\mathbb{N}$ such that $|x_m-p_{j_0}|=|z_m-p_{j_0}|=\eta$ for some $\eta\geq 0$ and all $m\geq m_0$. Now assume by contradiction that $\eta>0$. Observe that $x_{m}\in\partial I_{j_m}\cap\partial K_{n_m}$, so that by Lipschitz regularity, for $r=r(x_m)>0$ small enough, we can write 
\begin{equation*}
B_r(x_{m})=\Big(B_r(x_{m})\cap\overline{I_{j_m}}\Big)\cup \Big(B_r(x_{m})\cap \overline{K_{n_m}}\Big).
\end{equation*} 
Thus in one of the two sets $\overline{I_{j_m}}$ or $\overline{K_{n_m}}$ we can find an element $x$ such that $|x-p_{j_0}|<|x_{m}-p_{j_0}|=\eta$, which contradicts the definition of $x_m$ and $z_m$. Consequently $x_{m_0}=p_{j_0}$ as claimed. 

Having at hand the auxiliary components $I_{j_m}$ and $K_{n_m}$ we now prove iteratively that $a_{j_m}=a_{j_1}$ for all $m\leq m_0$ which yields the claim. First note that $x_1\in\partial I_{j_1}\cap\partial K_{n_1}$, so that there exists a Lipschitz curve $\gamma_1:(-\delta_1,\delta_1)\to\R^2$ such that $\gamma_1((-\delta_1,\delta_1))\subset\partial I_{j_1}\cap\partial K_{n_1}$. Since $a_{j_1}=0$, on some small ball $B_{r_1}(x_1)$ the function $\psi\in H^2_{\rm loc}(\R^2)$ solves the problem 
\begin{equation}\label{e.localPDE}
\begin{cases}
\triangle \psi = 0 &\mbox{on $B_{r_1}(x_1)$,} \\
\psi = b_{j_1}y_1+c_{j_1}y_2+d_{j_1}&\mbox{in $B_{r_1}(x_1)\cap I_{j_1}$.}\\
\end{cases}
\end{equation}
Thus by uniqueness $\psi = b_{j_1}y_1+c_{j_1}y_2+d_{j_1}$ on the connected set $K_{n_1}$ and the gradient is constant on the interior. We next transfer this information to $I_{j_2}$. To this end, observe that $z_1\in\partial K_{n_1}\cap\partial I_{j_2}$ and thus we find a non-constant Lipschitz-curve $\gamma_2:(-\delta_2,\delta_2)\to \R^2$ such that $\gamma_2((-\delta_2,\delta_2))\subset\partial I_{j_2}\cap\partial K_{n_1}$. For $\nabla\psi$ we obtain the compatibility condition
\begin{equation*}
\begin{pmatrix}
b_{j_1} \\ c_{j_1}\end{pmatrix} = 2a_{j_2}\gamma_2(t)+\begin{pmatrix}
b_{j_2} \\ c_{j_2}\end{pmatrix}
\end{equation*}
for every $t\in (-\delta_2,\delta_2)$. This is possible only if $a_{j_2}=0$. Now we repeat this reasoning until $x_n=p_{j_0}$ and conclude that $a_{j_0}=0$. 

We are now in the position to conclude assuming (C3-a). From the above argument we deduce that $a_j=0$ for all $j$. Hence equation \eqref{eq:THEPDE} implies that $\psi$ is harmonic on $\R^2$ and linear on an open set. Thus $\psi$ is globally linear and therefore integrating over $Q$ and taking the expectation of the equation $0=\nabla^2\psi = M+\nabla v$ yields first $M=0$ and then $\nabla v=0$ without the expectation.

Finally we treat Assumption~(C3-b). In this case we can use the same arguments except that in general $a_{j_1}\neq 0$. However, at all boundaries of the connected components $K_n$ of the set $\{x\,|\,\LL(x)=\LL^1\}$ the solution $\psi$ is explicit and the gradients are of the form \eref{formulagradient} for some $a_{j_n}$ and with $Q_1=I$. One then easily obtains the compatibility conditions $a_{j_1}=a_j$ for all $j$, so that the function $\nabla^2\psi=M+\nabla v$ is given by
\begin{equation*}
\nabla^2 \psi = \chi \begin{pmatrix} -2a_1 &0 \\ 0 &2a_1\end{pmatrix}+(1-\chi)\begin{pmatrix} 2a_1 &0 \\ 0 &2a_1\end{pmatrix} = \begin{pmatrix} 2a_1(1-2\chi) &0 \\ 0 &2a_1\end{pmatrix}.
\end{equation*}
By the ergodic theorem the expectation of the RHS matrix agrees with $M$ and so it can have at most rank 1. Note that $\nabla^2\psi$ is stationary and the entry $\partial_2^2\psi$ is invariant under spatial shifts. Ergodicity implies that $a_1$ is deterministic. Thus either $a_1=0$, $M=0$ and $\nabla v = 0$ as before, or, if $a_1\neq 0$, then $\expec{\chi}=\frac{1}{2}$ and $M= \kappa e_2\otimes e_2$ with $\kappa = 2a_1$ as claimed.
\end{proof}

\subsection{Proof of \pref{lambdaA}: Assumption~\ref{a.non-neg}}

We split the proof into two steps, first prove $\Lambda_6>0$, and then $\Lambda_4>0$.

\medskip

\noindent \textit{Step~1.} Proof of $\Lambda_6>0$.

\noindent We argue by contradiction and assume that $\Lambda_6=0$. 
Consider $\nabla v_n$ a minimizing sequence of stationary fields of $\R^2$ with $\expec{\int_Q |\nabla v_n|^2}=1$ such that
\begin{equation}\label{e.ass-to-zero}
\lim_{n\to \infty}\expec{\int_Q \nabla v_n\cdot \LL\nabla v_n}\,=\,0.
\end{equation}
By \lref{weaktozero} with $M=0$, $\nabla v_n\rightharpoonup 0$ in $L^2(\Omega\times Q)$,
and it remains to argue that the convergence is strong to get a contradiction.
To this aim, we shall prove that if \eqref{e.ass-to-zero} holds, then $(1-\chi)\nabla v_n$ converges strongly in $L^2(\Omega\times Q)$ to zero (recall that $\chi$ is the indicator function of the set $\{x\,|\,\LL(x)=\LL^1\}$).
Integrating \eqref{e.lowerbound} over the unit cube $Q$ yields in view of \eqref{e.ass-to-zero}
$$
\lim_{n\to \infty} \expec{{\int_{Q} (1-\chi) \Big(\big| \partial_2 v^1_n \big|^2+\big|\partial_1 v^2_n  \big|^2\Big)}}=0.
$$
Hence it is enough to prove that  $(1-\chi)\partial_1 v^1_n$ and
$(1-\chi) \partial_2 v^2_n$ converge strongly to zero in $L^2(\Omega\times Q)$ as well.
Since the argument is the same for both terms, we only treat $\partial_1 v^1_n$, and we follow the beginning of the argument of Step~2 in the proof of  \cite[Theorem~2.9]{Briane-Francfort-15}. Let $Z=\{\chi=1\}$ denote the (random) set of inclusions in $\R^2$. By assumption (A3), there exists $C<\infty$ such that for all $R\gg 1$ and
$Q_R=[-\frac R2,\frac R2)^d$, Ne\v{c}as' inequality \eqref{e.Neceq} holds in the form
\begin{equation}\label{uniKorn}
\int_{Q_R\setminus Z} \big|\partial_1 v^1_n \big|^2 \,\le \, \int_{M_R} \big| \partial_1 v^1_n \big|^2 \, \le \, C \| \nabla \partial_1 v^1_n \|_{H^{-1}(M_R)}^2
+C\frac{1}{|M_R|}\Big(\int_{M_R} \partial_1 v^1_n \Big)^2.
\end{equation}
We start with the control of the $H^{-1}$-norm, and appeal to \eqref{e.lowerbound} in the form
\begin{eqnarray*}
\lefteqn{ \| \nabla \partial_1 v^1_n \|_{H^{-1}(M_R)}^2 }
\\
&=& \|    \partial_1 \partial_1 v^1_n\|_{H^{-1}(M_R)}^2
 + \|  \partial_2 \partial_1 v^1_n \|_{H^{-1}(M_R)}^2
 \\
&\le&2 \|  \partial_1 (\partial_1 v^1_n - \partial_2 v^2_n)\|_{H^{-1}(M_R)}^2+ 2\|  \partial_2 \partial_1 v^2_n \|_{H^{-1}(M_R)}^2
+ \|   \partial_1 \partial_2 v^1_n \|_{H^{-1}(M_R)}^2
\\
&\le &2\|\partial_1 v^1_n -\partial_2 v^2_n \|_{L^2(M_R)}^2+ 2\|   \partial_1 v^2_n \|_{L^2(M_R)}^2
+ \|  \partial_2 v^1_n \|_{L^2(M_R)}^2
\\
&\le & \frac2\alpha \int_{Q_R\setminus Z} \nabla v_n\cdot \LL\nabla v_n+4\,\mu_1 \det \nabla v_n + 4\int_{Q_R^1\setminus Q_R}|\nabla v_n|^2
\\
&\le &  \frac2\alpha \int_{Q_R} \nabla v_n\cdot \LL\nabla v_n+4\,\mu_1 \det \nabla v_n + 4\int_{Q_R^1\setminus Q_R}|\nabla v_n|^2,
\end{eqnarray*}
by non-negativity of the first integrand.
Next we bound the last term in \eqref{uniKorn}, and note that
\begin{equation*}
\frac{1}{|M_R|}\Big(\int_{M_R}  \partial_1 v^1_n \Big)^2 \, \le \, \frac{2}{|Q_R\setminus Z|}\Big(\int_{Q_R\setminus Z} \partial_1 v^1_n \Big)^2+\frac{2}{|Q_R\setminus Z|}\Big(\int_{Q^1_R\setminus Q_R} | \partial_1 v^1_n|\Big)^2.
\end{equation*}
We then appeal to the ergodic theorem, which yields almost surely
\begin{eqnarray*}
\lim_{R\uparrow \infty} \frac1{R^d}\int_{Q_R\setminus Z} \big| \partial_1 v^1_n \big|^2&=& \expec{\int_{Q}(1-\chi) \big| \partial_1 v^1_n \big|^2},
\\
\lim_{R\uparrow \infty} \frac1{R^d} \int_{Q_R} \nabla v_n\cdot \LL\nabla v_n+4\,\mu_1 \det \nabla v_n &=& \expec{\int_{Q} \nabla v_n\cdot \LL\nabla v_n+4\,\mu_1 \det \nabla v_n}
\\
&\stackrel{\eqref{a+4}}{=}& \expec{\int_{Q} \nabla v_n\cdot \LL\nabla v_n},
\\
\lim_{R\uparrow \infty} \frac1{R^d} \int_{Q_R\setminus Z}  \partial_1 v^1_n &=&\expec{\int_Q (1-\chi) \partial_1 v^1_n },
\end{eqnarray*}
and
\begin{eqnarray*}
\lim_{R\uparrow \infty} \frac1{R^d} \int_{Q^1_R\setminus Q_R} | \partial_1 v^1_n|&=& \lim_{R\uparrow \infty} \frac1{R^d} \int_{Q^1_R\setminus Q_R} |\nabla v_n|^2 \;=\; 0,
\\
\lim_{R\uparrow \infty}\frac{|Q_R\setminus Z|}{R^d}&=&\expec{\int_Q(1-\chi)} \;>\;0,
\end{eqnarray*}
where in the fourth equality we used that $\lim_{R\uparrow \infty}|Q^1_R|/|Q_R|= 1$. Combined with \eqref{uniKorn}, these five convergences and the previous estimate imply
\begin{equation*}
 \expec{\int_{Q}(1-\chi) \big|\partial_1 v^1_n \big|^2}\,\le\,  \frac{2C}\alpha \expec{\int_{Q} \nabla v_n\cdot \LL\nabla v_n}
 +C\Big(\expec{\int_Q (1-\chi)  \partial_1 v^1_n}\Big)^2.
\end{equation*}
By the weak convergence of $\nabla v_n$ to zero and \eqref{e.ass-to-zero}, the two RHS terms vanish in the limit $n \uparrow \infty$, so that
\begin{equation*}
\lim_{n\uparrow \infty} \expec{\int_{Q}(1-\chi) \big|\partial_1 v^1_n \big|^2}\,=\,0,
\end{equation*}
as claimed. The same result holds for $\frac{\partial v^2_n}{\partial y_2}$ and we have thus proved
\begin{equation}\label{e.nablav_n1}
\lim_{n\uparrow \infty}   \expec{\int_{Q}(1-\chi) |\nabla v_n|^2}\,=\,0.
\end{equation}

\medskip

We are in the position to conclude. Since $(1-\chi)\det \nabla v_n\to 0$ in $L^1(\Omega \times Q)$ by  \eqref{e.nablav_n1} and $\expec{\int_Q\det \nabla v_n}\equiv 0$
by \eqref{a+4}, 
we have $\expec{\int_Q \chi \det \nabla v_n}\to 0$. 
Subtracting twice this quantity to the following consequence of \eqref{e.lowerbound}
$$
\expec{\int_Q \chi \Big(  (\partial_1 u_n^1 + \partial_2 u_n^2)^2+( \partial_2 u_n^1-\partial_1 u_n^2)^2 \Big) }\to 0
$$
yields 
$$
\lim_{n\uparrow \infty}\expec{\int_Q \chi |\nabla v_n|^2 }\to 0.
$$
Combined with \eqref{e.nablav_n1}, this implies
$\lim_{n\uparrow \infty} \expec{\int_Q |\nabla v_n|^2 }= 0$, which contradicts the assumption $\expec{\int_Q |\nabla v_n|^2}=1$.

\medskip

\noindent \emph{Step 2.} Proof of $\Lambda_4>0$. 

\noindent By Step~1 and  \tref{alaGMT}, for all $M\in \R^{2\times2}$
\begin{equation*}
M \cdot \LL_*M\,:=\,\inf_{u\in\Hhr^1_1}\expec{\int_Q (M+\nabla u)\cdot \LL(M+\nabla u)}.
\end{equation*}
Assume that there exists a rank-one matrix $M$ such that $M\cdot\LL_*M=0$.
We shall then prove that $M$ necessarily vanishes. 
If $M\cdot\LL_*M=0$, there exists some minimizing sequence $u_n\in\Hhr^1_1$ such that
\begin{equation}\label{e.minseqrand}
\lim_n  \expec{\int_Q (M+\nabla u_n) \cdot \LL(M+\nabla u_n)}\,=0.
\end{equation}
Let us prove that $\nabla u_n$ is bounded in $L^2(\Omega\times Q)$.
Indeed, if $\expec{\int_Q |\nabla u_n|^2} \to \infty$, then 
\begin{equation*}
v_n:=\frac1{\expec{\int_Q |\nabla u_n|^2}^\frac12} u_n
\end{equation*}
is well-defined for $n$ large enough and satisfies $\expec{\int_Q |\nabla v_n|^2}=1$. It then follows from \eqref{e.minseqrand} that
\begin{equation*}
\lim_n \expec{\int_Q \nabla v_n \cdot \LL \nabla v_n}\,=\,0,
\end{equation*}
which contradicts the fact that $\Lambda_6>0$. Therefore $\nabla u_n$ is bounded in $L^2(\Omega\times Q)$, and we are in the position to apply \lref{weaktozero} to $M$ and $\nabla u_n$. The latter then yields $M=0$, and therefore $\Lambda_4>0$.


\subsection{Proof of \pref{lambdaB}: Assumption~\ref{b.non-neg}} 

We split the proof into two steps, first prove $\Lambda_6>0$, and then $\Lambda_4>0$.

\medskip

\noindent \textit{Step~1} Proof of $\Lambda_6>0$.

\noindent As in the corresponding part of the proof of \pref{lambdaA}, we argue by contradiction and suppose that there is a sequence $v_n\in H^1_{\rm per}(Q,\R^2)$ such that $\int_Q|\nabla v_n|^2=1$ and
\begin{equation}\label{e.per-assump}
\lim_{n\to \infty}\int_Q\nabla v_n\cdot \LL\nabla v_n = 0.
\end{equation}
By \lref{weaktozero}, $\nabla v_n$ converges weakly to $0$ in $L^2(Q)$ and we shall obtain a contradiction 
by proving the strong convergence of $\nabla v_n$ to $0$. Again \eqref{e.lowerbound} combined with \eqref{e.per-assump} imply that on the set $\{x\,|\,\LL(x)\neq\LL^1\}$ the functions $\partial_2 v^1_n$ and $\partial_1 v^2_n$ converge to zero strongly in $L^2(Q)$. For the function $\partial_1 v_n^1$ we use Ne\v{c}as' inequality \eqref{e.Neceq} on each connected component $I_j$ of the set $\{x\,|\,\LL(x)\neq\LL^1\}$ that intersects $Q$. Since in each periodic cell we only have finitely many connected components, there is a uniform constant $C$ such that
\begin{equation*}
\int_{I_j}\big|\partial_1 v_n^1\big|^2\leq C\|\nabla \partial_1 v_n^1\|^2_{H^{-1}(I_j)}+C\Big(\int_{I_j} \partial_1 v_n^1 \Big)^2.
\end{equation*}
Arguing as in the proof of \pref{lambdaA}, we may control the $H^{-1}(I_j)$-norm of $\partial_1 v^1_n$ by 
\begin{equation*}
\| \nabla \partial_1 v^1_n\|_{H^{-1}(I_j)}^2 
\leq  \frac{2}{\alpha} \int_{I_j} \nabla v_n\cdot \LL\nabla v_n+4\,\mu_1 \det \nabla v_n.
\end{equation*}
Summing over $j$ we obtain
\begin{equation*}
\int_{\bigcup_j I_j}\big|\frac{\partial v_n^1}{\partial y_1}\big|^2\leq C \int_{Q} \nabla v_n\cdot \LL\nabla v_n+4\,\mu_1 \det \nabla v_n+C\sum_j\Big(\int_{I_j}\frac{\partial v_n^1}{\partial y_1}\Big)^2.
\end{equation*}
The first term on the RHS converges to zero by assumption \eqref{e.per-assump}, while the second term vanishes by the weak convergence $\nabla v_n\rightharpoonup 0$ since the sum is finite. The same argument can be repeated for the function $\frac{\partial v_n^2}{\partial y_2}$, so that $\nabla v_n$ converges to zero strongly on $\{x\,|\,\LL(x)\neq\LL^1\}$. 
We then conclude as in Step~1 of the proof of \pref{lambdaA}, obtain that $\nabla v_n\to 0$ strongly in $L^2(Q)$, which contradicts the normalization assumption.

\medskip

\noindent \textit{Step~2.} Proof of $\Lambda_4>0$.

\noindent  The proof of $\Lambda_4>0$, which solely relies on $\Lambda_6>0$,  \tref{alaGMT}, and  \lref{weaktozero},  is the same as for  \pref{lambdaA}. In particular, it also yields the desired implication $\Lambda_6>0\implies \Lambda_4>0$ in the random setting.


\subsection{Proof of \pref{lambdaC}: Assumption~\ref{c.non-neg}} 

We first prove the implication $\Lambda_6>0\implies\Lambda_4>0$ under assumption (C3-a) and then the additional property $\Lambda_6>0$ when $\LL$ is periodic. As a final step we treat the assumption (C3-b).

\medskip

\noindent \textit{Step~1} Proof of the implication $\Lambda_6>0\implies\Lambda_4>0$.

\noindent The argument is identical to {\it Step 2} in the proof of \pref{lambdaA} as it is only based on the fact $\Lambda_6>0$, \tref{alaGMT} and \lref{weaktozero}

\medskip

\noindent \textit{Step~2.} Proof of $\Lambda_6>0$ in the periodic case.

\noindent The argument is almost identical to the corresponding one for \pref{lambdaB}. Again we argue by contradiction and suppose that there is a sequence $v_n\in H^1_{\rm per}(Q,\R^2)$ such that $\int_Q|\nabla v_n|^2=1$ and
\begin{equation}\label{e.per-assumpC}
\lim_{n\to \infty}\int_Q\nabla v_n\cdot \LL\nabla v_n = 0.
\end{equation}
By \lref{weaktozero}, $\nabla v_n$ converges weakly to $0$ in $L^2(Q)$. We argue that $\nabla v_n$ to $0$ strongly in $L^2(Q)$. Note that \eqref{e.lowerbound} combined with \eqref{e.per-assumpC} imply that on the set $\{x\,|\,\LL(x)\neq\LL^1\}$ the functions $\partial_2 v^1_n$ and $\partial_1 v^2_n$ converge to zero strongly in $L^2(Q)$. For the function $\partial_1 v_n^1$ we want again use Ne\v{c}as' inequality \eqref{e.Neceq}. However now the connected components {may be unbounded. We then apply \lref{surgery} to truncate each connected component $I_j$ in such a way that the new set $I_j^{\prime}$ is bounded with Lipschitz boundary and agrees with $I_j$ on $Q$. In particular, it has finitely many connected components that we denote by $I^{\prime}_{jk}$.}  Hence, by Ne\v{c}as' inequality {on each $I_{jk}$}, there is a uniform constant $C$ such that
\begin{equation*}
\int_{I^{\prime}_{{jk}}}\big|\partial_1 v_n^1\big|^2\leq C\|\nabla \partial_1 v_n^1\|^2_{H^{-1}(I^{\prime}_{{jk}})}+C\Big(\int_{I^{\prime}_{{jk}}} \partial_1 v_n^1 \Big)^2.
\end{equation*}
Now the arguments are the same as for \pref{lambdaB} and we leave the details to the reader.

\medskip

\noindent \textit{Step~3.} Proofs under assumption (C3-b).

\noindent We start with the case $\theta_1=\frac12$. Then we consider the matrix $M=e_2\otimes e_2$ and the non-zero field $U\in\Hhr^0_1$ defined by 
\begin{equation*}
U(x,\omega)=(1-2\chi(x_1,\omega))e_1\otimes e_1.
\end{equation*}
A straightforward calculation based on \eqref{eq:cond} and \eqref{a+5} yields 
\begin{equation*}
\expec{\int_Q (M+U)\cdot\LL (M+U)}=\expec{\int_Q \chi 4\mu_1+(1-\chi)4(\mu_2+\lambda_2)}=\theta_1 4\mu_1+(1-\theta_1)4(\lambda_2+\mu_2)=0.
\end{equation*}
To conclude that $\Lambda_4=0$, it suffices to find a sequence $v_n\in\Hhr^1_1$ such that $\nabla v_n\to U$ in $L^2(Q\times \Omega)$. Since this property is well-known, we just sketch the argument: for $T\gg 1$ we consider the weak PDE formulation
\begin{equation}\label{e.weakinprob}
\expec{\int_Q \frac{1}{T}v\cdot\varphi+(\nabla v-U)\cdot\nabla\varphi}=0
\end{equation} 
for all $\varphi\in\Hhr^1_1$. By the Lax-Milgram theorem there exists a unique solution $v^T\in\Hhr^1_1$ to the above problem. Testing the equation with $\varphi = v^T$, we deduce that $T^{-\frac{1}{2}}v^T$ and $\nabla v^T$ are bounded sequences in $L^2(Q\times\Omega)$. Thus we can assume that $\nabla v^T\rightharpoonup Z$ for some $Z\in L^2(Q\times\Omega)$ (so that $\expec{\int_Q Z}=0$) and moreover we can pass to the limit in the equation to deduce that for all $\varphi\in\Hhr^1_1$
\begin{equation*}
\expec{\int_Q (Z-U)\cdot\nabla\varphi}=0.
\end{equation*}
Arguing as for \eqref{a+1}, on the one hand we know that ${\rm curl}\,Z={\rm curl}\,U=0$ almost surely. On the other hand, the above equality implies that ${\rm div}(Z-U)=0$ almost surely by the same reasoning. Combining the div-curl-Lemma on $Q$ and the ergodic theorem for the stationary function $Z-U$, we deduce
\begin{align*}
\expec{\int_Q|Z-U|^2}&=\lim_{R\uparrow \infty}\frac{1}{R^d}\int_{Q_R}|Z-U|^2=\lim_{\e\downarrow 0}\int_Q|Z(x/\e,\omega)-U(x/\e,\omega))|^2
\\
&=\int_Q\Big|\expec{\int_Q Z-U}\Big|^2\,=\,0,
\end{align*}
where we used that $\expec{\int_Q U}=\expec{\int_Q Z}=0$. This shows that $Z=U$ in $L^2(Q\times \Omega)$.
We finish the argument proving strong convergence of $\nabla v^T$. Testing \eref{weakinprob} with $v^T$ itself, we obtain from the weak convergence $\nabla v^T\rightharpoonup U$, the weak lower-semicontinuity of the norm, and the equation 
\begin{equation*}
\expec{\int_Q |U|^2}\le \liminf_{T\to\infty}\expec{\int_Q\left|\nabla v^T\right|^2}\leq \lim_{T\to\infty}\expec{\int_Q U\cdot\nabla v^T}=\expec{\int_Q|U|^2}.
\end{equation*}
The strong convergence is now a consequence of the Hilbert space structure of $L^2(Q\times\Omega,\R^{2\times 2})$.

In the case $\theta_1\ne\frac{1}{2}$, the implication $\Lambda_6>0\implies\Lambda_4>0$ is again a simple consequence of \tref{alaGMT} and \lref{weaktozero} (see again {\it Step 2} in the proof of \pref{lambdaA} for the details).

\subsection{Proof of \cref{hom-homC}: Assumption~(C3-b)}
The results under assumption (C3-a) are a direct consequence of \tref{alaGMT} and \pref{lambdaC}. However, given assumption (C3-b), the positivity of $\Lambda_6$ is not known in the random setting, so we prove the statements directly. First note that the case of volume fraction $\theta_1=\frac12$ was implicitly proven in \pref{lambdaC}, where we showed that $\LL_*$ loses ellipticity in the direction $M=e_2\otimes e_2$. Hence it remains to show that $\LL_*$ is strictly strongly elliptic whenever $\theta_1\neq\frac12$. We argue by contradiction. Due to \pref{LAMBDA} and \rref{BraidesBriane} homogenization holds with $\LL_*$ being characterized by the formula \eref{asfo-JKO}, so let us assume that there exists a rank one matrix $M=a\otimes b$ and a sequence $v_n\in\Hhr^1_1$ such that, taking into account \eqref{a+2},
\begin{equation}\label{e.lossofellipticity}
\lim_{n\to\infty}\expec{\int_Q (M+\nabla v_n)\cdot\LL(M+\nabla v_n)+4\mu_1\det(M+\nabla v_n)}=0.
\end{equation}  
We denote by $\mathcal{L}_M(\nabla v)$ the functional in the above equation. We shall construct a matrix-field with the same energy that has a one-dimensional profile. Given $k\in\mathbb{N}$ define $w^k_n:(0,1)\times\Omega\to\R^{2\times 2}$ as
\begin{equation*}
w_n^k(x_1,\omega)=\frac{1}{2k}\int_{-k}^k\nabla v_n(x_1,x_2,\omega) dx_2.
\end{equation*}
Since the integrand in \eref{lossofellipticity} is convex in the argument $M+\nabla v_n$, stationarity, Fubini's theorem and Jensen's inequality imply that for all $k$
\begin{align*}
\mathcal{L}_M(\nabla v_n)=\expec{\int_0^1\frac{1}{2k}\int_{-k}^k(M+\nabla v_n)\cdot\LL(M+\nabla v_n)+4\mu_1\det(M+\nabla v_n)dx_2dx_1}\geq\mathcal{L}_M(w_n^k),
\end{align*}
where we used that $\LL$ depends only on $x_1$. Moreover, for fixed $n$, due to stationarity of $\nabla v_n$ and Fubini's theorem, the sequence $w_n^k$ is bounded in $L^2((0,1)\times \Omega)$, so that without loss of generality there exists a weak limit $w_n\in L^2((0,1)\times\Omega)$. Using again convexity of $\mathcal{L}_M$, weak lower semicontinuity yields $\mathcal{L}_M(\nabla v_n)\geq\mathcal{L}_M(w_n)$. Next we identify some of the components of $w_n$. To this end, let $\eta_k$ be a smooth cut-off function such that $\eta_k\equiv 1$ on $[-k,k]$ with support in $[-k-1,k+1]$ and $\|\eta^{\prime}_k\|_{\infty}\leq 2$. Then by stationarity of $v_n$ and $\nabla v_n$
\begin{align*}
\expec{\int_0^1\left|(w^k_n)_{12}\right|^2}&\leq \expec{\int_0^1 2\left|\frac{1}{2k}\int_{\R}\eta_k(x_2)\partial_2v_n^1\right|^2+\frac{1}{k^2}\left(\int_{-k-1}^{-k}|\partial_2v_n^1|^2+\int_{k}^{k+1}|\partial_2v_n^1|^2\right)}
\\
&\leq 2\expec{\int_0^1\left|\frac{1}{2k}\int_{\R}\eta_k^{\prime}(x_2)v_n^1\right|^2}+\frac{2}{k^2}\expec{\int_Q|\partial_2v_n^1|^2}\leq \frac{4}{k^2}\expec{\int_Q|v_n^1|^2+|\partial_2 v_n^1|^2}.
\end{align*}
Hence $(w_n^k)_{12}\to 0$ strongly in $L^2((0,1)\times\Omega)$. A similar calculation holds for the component $(w_n^k)_{22}$, so that the weak limit $w_n$ has non-zero entries only in its first column.  This structure combined with the lower bound \eref{lowerbound} yields
\begin{align*}
\mathcal{L}_M(\nabla v_n)\geq\mathcal{L}_M(w_n)
\geq& \alpha\expec{\int_Q\chi (M_{11}+(w_n)_{11}+M_{22})^2+\chi(M_{12}-M_{21}-(w_n)_{21})^2}
\\
&+\alpha\expec{\int_Q(1-\chi) \left((M_{11}+(w_n)_{11}-M_{22})^2+M_{12}^2+(M_{21}+(w_n)_{21})^2\right)}.
\end{align*}
Since the LHS converges to zero when $n\to\infty$, we infer on the one hand that $w_n$ converges strongly to some matrix-field $w\in L^2((0,1)\times\Omega)$ which is nonzero only in the first column. On the other hand, $M_{12}=0$ and therefore $M_{11}M_{22}=0$ due to the rank one assumption. Using this property, the limit component $w_{11}$ can be rewritten as
\begin{equation*}
w_{11}=-\chi(M_{11}+M_{22})+(1-\chi)(M_{22}-M_{11})=-M_{11}+(1-2\chi)M_{22}
\end{equation*}
Since by construction $\expec{\int_Qw}=0$, this equality and the fact that $\theta_1\neq\frac12$ imply $M_{11}=M_{22}=0$. For the last component, we note that
\begin{equation*}
w_{21}=-\chi M_{21}-(1-\chi)M_{21}=-M_{21},
\end{equation*}
so again integrating over $Q$ and taking the expectation shows that $M_{21}=0$. Thus $M=0$ and we reached a contradiction.
\appendix
\section{Ne\v{c}as inequality}
We quickly argue that under appropriate geometric properties on the microstructure, assumption (A3) holds true.
\begin{lemma}\label{l.example}
In addition to conditions (A1) and (A2), assume that there exist finitely many open, connected sets $S_1,\dots,S_N\subset\R^2$ with Lipschitz boundary such that for each connected component $I_j$ of $\{x\,|\,\LL(x)=\LL^1\}$ there exist a rigid motion $Q_i$ and $j(i)\in\{1,\dots,N\}$ such that $I_i=Q_iS_{j(i)}$. Moreover assume that the connected components are well-separated in the sense that 
\begin{equation}
\dist(I_i,I_j)\geq c>0\quad\quad \text{for all }i\neq j.  
\end{equation} 
Then the Ne\v{c}as inequality (A3) holds.
\qed
\end{lemma}
\begin{remark}\label{r.uniformestimates}
The same conclusions hold if we allow for infinitely many different shapes $S_j$ provided all constants appearing in the proof below are uniform with respect to the inclusions. 
\qed
\end{remark}
\begin{proof}[Proof of \lref{example}]
We essentially follow the arguments of \cite{Allaire1991} and first construct a linear, continuous operator $\mathcal{H}_R:H^1_0(Q^1_R)^d\to H^1_0(M_R)^d$ with equibounded (in $R$) operator-norm that preserves the divergence and acts as the identity on $H^1_0(M_R)\subset H^1_0(Q^1_R)$. To this end, we first consider a fixed shape $S=S_j$. Then there exists $0<\delta<\min\{c/2, C_1\}$ such that the set $S^{\delta}=(S+B_{\delta})\setminus S$ has also Lipschitz boundary (here $C_1$ is given by assumption (A2)). On this ``safety zone'' we consider the following Stokes problem: Given $\bar{u}\in H^1(S+B_{\delta})$, find $v\in H^1(S^{\delta})^d$ and $q\in L^2(S^{\delta})$ with $\int_{S^{\delta}}q=0$ solving the stationary Stokes equation
\begin{equation}\label{e.stokes}
\begin{split}
-\triangle v+\nabla q = -\triangle \bar{u}\quad\quad\text{ on }S^{\delta};
\\
{\rm div}\,v = {\rm div}\,\bar{u} + \frac{1}{|S^{\delta}|}\int_{S}{\rm div}\,\bar{u}\,\mathrm{d}x\quad\quad\text{ on }S^{\delta};
\\
v = 0 \quad\quad\text{ on }\partial S,\quad\quad\quad v=\bar{u} \quad\quad\text{ on }\partial S^{\delta}\setminus\partial S.
\end{split}
\end{equation}
Using the divergence theorem,  the compatibility condition
\begin{align*}
\int_{S^{\delta}}{\rm div}\,v dx=\int_{\partial S^{\delta}} v \cdot n d\sigma 
\end{align*}
holds.
Such a system admits a unique weak solution $(v,q)$ which satisfies the a priori estimate
\begin{equation*}
\|v\|_{H^1(S^{\delta})}+\|q\|_{L^2(S^{\delta})}\leq C\big(\|-\triangle v+\nabla q\|_{H^{-1}(S^{\delta})}+\|{\rm div}\,v\|_{L^2(S^{\delta})}+\|v\|_{H^{\frac{1}{2}}(\partial S^{\delta})}\big),
\end{equation*}
where $C$ depends on $d$ and $S^{\delta}$. The first RHS term is bounded by $\|\nabla \bar{u}\|_{L^2(S^{\delta})}$, while, by Jensen's inequality, the second RHS term can be controlled by 
\begin{equation*}
\|{\rm div}\,v\|_{L^2(S^{\delta})}\leq \|\nabla \bar{u}\|_{L^2(S^{\delta})}+\Big(\frac{|S|}{|S^{\delta}|}\Big)^{\frac{1}{2}}\|\nabla \bar{u}\|_{L^2(S)}.
\end{equation*}
For the third RHS term, we introduce a cut-off function $\chi\in C^1(\R^d,[0,1])$ such that $\chi \equiv 0$ on $S^{\delta/4}$ and $\chi\equiv 1$ on $\R^d\setminus S^{3\delta/4}$, and $\|\nabla\phi\|_{\infty}\leq 4\delta^{-1}$. Then by trace estimates, there exists a constant $C(S,\delta)<\infty$ such that we have
\begin{equation*}
\|v\|_{H^{\frac{1}{2}}(\partial S^{\delta})}\leq C(S,\delta) \|\chi \bar{u}\|_{H^1(S^{\delta})}\leq C(S,\delta)\Big( \|\bar{u}\|_{H^1(S^{\delta})}+4d\delta^{-1}\|\bar{u}\|_{L^2(S^{\delta})}\Big).
\end{equation*}
We then conclude that there exists a (possibly different) constant $C(S,\delta)<\infty$ depending only on $S$ and $\delta$ such that
\begin{align}\label{e.apriori}
\|v\|_{H^1(S^{\delta})}+\|q\|_{L^2(S^{\delta})}\leq C(S,{\delta})\|\bar{u}\|_{H^1(S+B_{\delta})}.
\end{align}
We are now in the position to define the operator  $\mathcal{H}_R$. 
For all inclusions $I_i$, we write the associated rigid body motion as $Q_i:x\mapsto D_ix+c_i$ with $D_i\in SO(2)$. 
Let $u\in H^1_0(Q^1_R)^d$.
For all $i$ such that $I_i \subset Q^1_R$ we denote by $v_{i}\in H^1(S_{j(i)}^{\delta_i})$ the unique solution of the Stokes problem \eqref{e.stokes} on $S_{j(i)}$ with $\bar{u}=(D_i^{-1}u)\circ Q_{i}$. 
We then define
\begin{equation*}
\mathcal{H}_R(u)=
\begin{cases}
(D_i v_{i})\circ Q_i^{-1} &\mbox{if $x\in I_j+B_{\delta}$ for some $i$,}
\\
u &\mbox{otherwise in $M_R$.}
\end{cases}
\end{equation*}
Since \eqref{e.stokes} is a linear system, $\mathcal{H}_R$ is defined as a linear operator from $H^1_0(Q^1_R)^d$ to $H^1_0(M_R)^d$. In addition, by \eqref{e.apriori} and (A2) 
we infer that
\begin{equation*}
\|\mathcal{H}_R(u)\|^2_{H_0^1(M_R)}\leq \|u\|^2_{H_0^1(Q^1_R)}+\max_{1\leq i\leq N}C(S_i,{\delta_i})^2\sum_{I_{i}\subset Q^1_R}\|u\|^2_{H^1(I_{i}+B_{\delta})}
\leq C\|u\|^2_{H_0^1(Q_R^1)}
\end{equation*}
for some $C$ independent of $R$. Next, if $u=0$ on $\partial I_i$ for some  inclusion $I_i$, then by uniqueness of solutions
of the Stokes system, it follows that $v_{i}=\bar{u}$ on $S^{\delta}_{i}$. Hence $\mathcal{H}_R(u)=u$ whenever $u\in H^1_0(M_R)^d$. Finally, consider the case ${\rm div}\,u=0$. Since $D_i\in SO(2)$, we deduce that for $y=Q_i^{-1}x\in S_i^{\delta_i}$ it holds that
\begin{align*}
{\rm div_x}\,(D_iv_i(Q_i^{-1}x))&={\rm div_y}\,(D_i^{-1}D_iv(y))={\rm div_y}\,v_i(y)
\\
&={\rm div_y}\,(D_i^{-1}u(Q_iy))+\frac{1}{|S^{\delta_i}_i|}\int_{S_i}{\rm div_y}\,(D_i^{-1}u(Q_iy))\,\mathrm{d}y
\\
&={\rm div_x}\,u(x)+\frac{1}{|S^{\delta_i}_i|}\int_{I_i}{\rm div_x}\,u(x)\,\mathrm{d}x=0,
\end{align*}
which shows ${\rm div}\,\mathcal{H}_R(u)=0$. 
Hence the operator $\mathcal{H}_R$ defined above has the desired properties. 

\medskip

We now conclude.
Since $M_R$ is a bounded, open, connected set with Lipschitz boundary, \cite[Proposition 1.1.4]{Allaire1991} yields (after rescaling) an extension operator $\mathcal{E}_R:L^2(M_R)/\R\to L^2(Q^1_R)/ \R$ that satisfies
\begin{equation*}
\|\nabla\mathcal{E}_R(u)\|_{H^{-1}(Q^1_R)}\leq C\|\nabla u\|_{H^{-1}(M_R)} 
\end{equation*}
for some constant $C$ independent of $R$. 
This implies the desired inequality \eqref{e.Neceq} in (A3) for all $u\in L^2(M_R)/\R$ in the form
\begin{equation*}
\|u\|_{L^2(M_R)}\leq \|\mathcal{E}_R(u)\|_{L^2(Q^1_R)}\leq C \|\nabla \mathcal{E}_R(u)\|_{H^{-1}(Q^1_R)}\leq C\|\nabla u\|_{H^{-1}(M_R)},
\end{equation*}
where we used that Ne\v{c}as' inequality indeed holds on cubes: there exists $C<\infty$ such that for all $R>0$ and $w\in L^2(Q^1_R)$ we have
$$
\|w\|_{L^2(Q^1_R)}\leq C \|\nabla w\|_{H^{-1}(Q^1_R)}.
$$
\end{proof}


\section{Measurable selections}\label{append:selection}

We briefly sketch the proof of the measurable selection results used in {\it Step 2} of the proofs of \lref{Lambda123} and \pref{order}. For $L>0$ we set $Y_L=\{u\in C^{\infty}(\R^d,\R^d):\,{\supp}(u)\subset \overline{B_L}\}$, which can be seen as a closed subspace of the Schwartz class. Hence $Y_L$ itself is a complete separable metric space. We now implicitly fix $L$ and a small number $\delta>0$.
Define the set-valued function $\Gamma:\Omega\to\mathcal{P}(Y_L)$ by
\begin{equation*}
\Gamma(\omega):=\Big\{u\in Y_L:\,\int_{\R^d}\nabla u\cdot\LL(x,\omega)\nabla u \,dx <g(\omega)\int_{\R^d}|\nabla u|^2 \,dx\Big\},
\end{equation*}
where the function $g:\Omega\to\R$ is defined by
\begin{equation*}
g(\omega)=\inf_{u\in Y_L}\frac{\int_{\R^d}\nabla u\cdot\LL(x,\omega)\nabla u\,dx}{\int_{\R^d}|\nabla u|^2\,dx}+\delta.
\end{equation*}
Note that by separability of $Y_L$ and joint measurability of $\LL$, the function $g$ is measurable since we can take the infimum in its definition over a countable set. Moreover, since $\delta>0$ we have $\Gamma(\omega)\neq\emptyset$ for all $\omega\in\Omega$. We argue that the graph of $\Gamma$ defined by
\begin{equation*}
{\rm Gr}(\Gamma)=\{(\omega,u)\in\Omega\times Y_L:\,u\in\Gamma(\omega)\}
\end{equation*}
belongs to the product $\sigma$-algebra $\mathcal{F}\otimes\mathcal{B}(Y_L)$, where $\mathcal{B}(Y_L)$ denotes the Borel sets on $Y_L$. To this aim, 
we observe that the function $H:\Omega\times Y_L\to\R$ defined by
\begin{equation*}
H(\omega,u)=\int_{\R^d}\nabla u\cdot\LL(x,\omega)\nabla u \,dx -g(\omega)\int_{\R^d}|\nabla u|^2 \,dx
\end{equation*}
is a Carath\'eodory function by the joint measurability of $\LL$ and Fubini's theorem. 
Since $Y_L$ is separable, this implies the joint measurability of $H$. Hence, ${\rm Gr}(\Gamma)=H^{-1}((-\infty,0))\in\mathcal{F}\otimes \mathcal{B}(Y_L)$, 
as claimed. 
By assumption $(\Omega,\mathcal{F},\mathbb{P})$ is a complete finite measure space. Therefore we may apply Aumann's measurable selection theorem (see \cite[Theorem 18.26]{AliBo}) to infer that there exists a $\mathcal{F}-\mathcal{B}(Y_L)$-measurable selection $\omega\mapsto u(\omega)\in \Gamma(\omega)$. Due to separability of $Y_L$, the function $\omega\mapsto u(\omega)$ is also strongly measurable. Moreover, by continuity also the real-valued function
\begin{equation*}
\omega\mapsto \int_{\R^d}|\nabla u(\omega,x)|^2\,dx
\end{equation*}  
is measurable. Since we assume that $\Lambda\geq 0$ and $\delta>0$, we have $0\notin \Gamma(\omega)$ for all $\omega\in\Omega$, so that the function
\begin{equation*}
\omega\mapsto v(\omega):=\frac{u(\omega)}{\int_{\R^d}|\nabla u(\omega,x)|^2\,dx}\in Y_L
\end{equation*}
is well-defined, strongly measurable, and by Poincar\'e's inequality on $B_L$ we deduce
\begin{equation*}
\expec{\int_{\R^d}|v(\omega,x)|^2+|\nabla v(\omega,x)|^2\, dx}\leq (C_L+1).
\end{equation*}
Observe that any Dirac mass belongs to the dual of $Y_L$, so that the function $(\omega,x)\mapsto v(\omega,x)$ is also of Carath\'eodory-type, whence jointly measurable. Finally, the definition of $\Gamma(\omega)$ yields the desired estimate
\eref{meas-selection}.

{
\section{Unbounded Lipschitz sets in the plane}
In this part of the appendix we prove the following auxiliary result that was needed in {\it Step 2} of the proof of \pref{lambdaC}.
\begin{lemma}\label{l.surgery}
Let $D\subset\R^2$ be an open set with Lipschitz boundary. For any bounded set $B\subset\R^2$ there exists a bounded open set $D_B\subset\R^2$ with Lipschitz boundary such that $D\cap B=D_B\cap B$.	
\end{lemma}
\begin{proof}
Let $R_0\gg 1$ be such that $B\subset Q_{R_0/2}$ and define $p_j:\R^2\to\R$ as the projection $p_j(x)=\langle x,e_j\rangle$. Observe that by definition the set $E:=\partial D\cap Q_{2R_0}$ is countably $\mathcal{H}^1$-rectifiable, that is, $\mathcal{H}^1$-measurable (since it is closed) and, up to $\mathcal{H}^1$-null sets, it can be covered by countably many Lipschitz graphs (since $D$ has Lipschitz boundary). By the generalized area-formula (see \cite[Theorem 2.91]{AFP}) it holds that
\begin{equation*}
\int_{\R}\mathcal{H}^0(E\cap p^{-1}_j(y))\,dy\leq \mathcal{H}^1(E)<+\infty,
\end{equation*}
where $\mathcal{H}^0$ denotes the counting measure. In particular, for almost every $y\in\R$ the cardinality of the set $\partial D\cap Q_{2R_0}\cap p_j^{-1}(y)$ is finite. Thus there exists $R\in (R_0,2R_0)$ such that $\partial Q_R\cap\partial D$ has finite cardinality and moreover, no point of $\partial Q_R\cap\partial D$ is contained in the four corners of $\partial Q_R$. Fix $x_0\in\partial Q_R\cap\partial D$. Depending on the local structure around $x_0$, we will now modify $Q_R$ to obtain the desired set by intersection. Since $D$ has Lipschitz boundary, by definition there exist a one-dimensional affine subspace $H\subset\R^2$ with $x_0\in H$, a Lipschitz function $g:H\to\R$ and $r,h>0$ such that
\begin{equation*}
\begin{split}
D\cap C_{r,h}(x_0,H)&=\{x+yn:\,x\in H\cap B_r(x_0),\,-h<y<g(x)\},
\\
\partial D\cap C_{r,h}(x_0,H)&=\{x+yn:\,x\in H\cap B_r(x_0),\,y=g(x)\},
\end{split}
\end{equation*}
where $C_{r,h}(x_0,H)=\{x+yn:\,x\in H\cap B_r(x_0),\,-h<y<h\}$ is a cylinder and $n$ is a unit normal vector to $H$. In what follows we assume without loss of generality $x_0=0$.

We first show that, up to reducing $r$ and $h$, we can slightly vary $H$ and $n$. To this end, we let $0<\delta\ll 1$ and consider a rotation $R_{\delta}\in SO(2)$ such that $\|R_{\delta}-I\|<\delta$. Define  $n_{\delta}=R_{\delta}n$ and $H_{\delta}$ as the one-dimensional affine subspace which is orthogonal to $n_{\delta}$ and contains $x_0$. We define a function $g_{\delta}:H_{\delta}\to\R$ as follows: For $v\in H_{\delta}$ and $t\in\R$ set
\begin{equation*}
h(t,v)=\langle v+tn_{\delta},n\rangle-g(P_H(v+tn_{\delta})),
\end{equation*}
where $P_H$ denotes the projection onto $H$. We claim that there exists a unique $t(v)\in\R$ such that $h(t(v),v)=0$. To this end, we show that $t\mapsto h(t,v)$ is uniformly monotone for $\delta$ small enough. Indeed, for all $t_2>t_1$ and $\delta=\delta(g)$ small enough we have
\begin{align*}
h(t_2,v)-h(t_1,v)&= (t_2-t_1)\langle n_{\delta},n\rangle-g(P_H(v+t_2n_{\delta}))+g(P_H(v+t_1n_{\delta}))
\\
&\geq (t_2-t_1)(1-\delta)-{\rm Lip}(g)(t_2-t_1)P_H(n_{\delta})\geq (t_2-t_1)(1-\delta-{\rm Lip}(g)\delta)
\\
&\geq\frac{1}{2}(t_2-t_1),
\end{align*}
so that for all $v\in H_{\delta}$ there exists a unique $t(v)$ such that $h(t(v),v)=0$. In addition, two zeros $(t_1,v_1)$ and $(t_2,v_2)$ satisfy
\begin{align*}
-|v_1-v_2|+(1-\delta)|t_1-t_2|&\leq \langle v_1-v_2+(t_1-t_2)n_{\delta},n\rangle=g(P_H(v_1+t_1n_{\delta}))-g(P_H(v_2+t_2n_{\delta}))
\\
&\leq{\rm Lip}(g)(|v_1-v_2|+\delta|t_1-t_2|)
\end{align*}
which implies that the mapping $H_{\delta}\ni v\mapsto t(v)$ is Lipschitz continuous. Next we argue that this function can be used in the definition of Lipschitz boundaries. Fix $0<h^{\prime}<h/2$. If $v+tn_{\delta}\in C_{\delta,h^{\prime}}(x_0,H_{\delta})$, then for $\delta$ small enough \begin{align*}
|P_H(v+tn_{\delta})|&\leq |v|+ h^{\prime}\delta\leq (1+h^{\prime})\delta<r,
\\
|\langle v+tn_{\delta},n\rangle|&\leq |v|\delta +h^{\prime}\leq \delta^2+\frac{h}{2}<h,
\end{align*}
where we used that $x_0=0$. In particular, we deduce from the properties of $g$ in $C_{r,h}(x_0,H)$ that for such $v+tn_{\delta}$ it holds that
\begin{equation*}
v+tn_{\delta}\in D\iff \langle v+tn_{\delta},n\rangle<g(P_H(v+tn_{\delta}))\iff h(t,v)<0\iff t<t(v),
\end{equation*}
where the last equivalence follows from the monotonicity of $t\mapsto h(t,v)$. If $v+tn_{\delta}\in\partial D$, then the above equivalences hold with equalities. We infer the claimed representation
\begin{equation*}
\begin{split}
D\cap C_{\delta,h^{\prime}}(x_0,H_{\delta})&=\{v+tn_{\delta}:\,v\in H_{\delta}\cap B_{\delta}(x_0),\,-h^{\prime}<t<t(v)\},
\\
\partial D\cap C_{\delta,h^{\prime}}(x_0,H_{\delta})&=\{v+tn_{\delta}:\,v\in H_{\delta}\cap B_{\delta}(x_0),\,t=t(v)\}.
\end{split}
\end{equation*}

Using the above result, we can assume that the segment of $\partial Q_R$, which contains $x_0$, is not orthogonal to the hyperplane $H$. Hence, up to a rigid motion, there exists some $a\in\R$ such that for some appropriate $r,h>0$ there are two possibilities:
\begin{itemize}
\item[(i)] $D\cap Q_R\cap C_{r,h}(x_0,H)=\{(x,y):\,|x|<r,\,-h<y<g(x),\,y<ax\}$,
\item[(ii)] $D\cap Q_R\cap C_{r,h}(x_0,H)=\{(x,y):\,|x|<r,\,-h<y<g(x),\,y>ax\}$.
\end{itemize} 
In case (i), we can replace $g$ by the Lipschitz function $\tilde{g}(x)=\min\{g(x),ax\}$, so that the resulting boundary around $x_0$ is still locally the graph of a Lipschitz function. It remains to treat case (ii), where we locally modify the cube $Q_R$. For $x\in (-r,r)\setminus\{0\}$ we may assume that $g(x)\neq ax$. If $g(x)\leq ax$ on $(-r,r)$, then the set in (ii) is empty and we are done. If $g(x)\geq ax$ on $(-r,r)$, then we replace $Q_R$ by $Q_R\cup Q_{\e}(x_0)$ for some positive $\e\ll r$. This new set has again Lipschitz boundary and its boundary intersects $\partial D$ in the same points as $\partial Q_R$ except $x_0$. Finally, we consider the case when $g(x)-ax$ changes sign. The idea is to cut out a small piece of $Q_R$. We set $\sigma={\rm sgn}(g(r/2)-ar/2)\in\{\pm 1\}$. In the chosen local frame we define a pentagon by the following corner points: Let $q_1=0$, $q_2=h/2e_2$ and then we take some positive $\e\ll h$ such that $g(x)<h/2$ for all $x\in[0,\sigma\e]$ (recall that $g(0)=0$). Given such $\e$, we define $q_3=(\sigma\e,h/2)$ and $q_4=(\sigma\e,g(\sigma \e))$. By construction $g(\sigma\e)>a\sigma \e$, so that by Lipschitz continuity of $g$ we can find a large slope $b>0$ such that the affine function $s:\R\to\R$ given by $s(x)=\sigma b x+g(\sigma\e)-b\e$ satisfies $s(\sigma\e)=g(\sigma\e)$ and the following two properties:
\begin{align*}
&s(x)<g(x)\text{ for all }x\in (-\e,\e),
\\
& s(x_*)=ax_*\text{ for some $x_*\neq 0$ such that }{\rm sgn}(x_*)=\sigma.
\end{align*}
We then define $q_5=x_*$. Connecting consecutively the points $p_1$ to $p_5$, which are given by transforming back the local frame of the points $q_i$, we obtain an open pentagon $P$, which is contained in $Q_R$. We then replace $Q_R$ by the Lipschitz set $Q_R\setminus \overline{P}$.

Performing the different operations on all $x_0\in\partial D\cap \partial Q_R$, we obtain a new bounded Lipschitz set $S_R$. We then set $D_R=D\cap S_R$. By construction it holds that $D\cap B=D_R\cap B$. Clearly $D_R$ is open and bounded. It remains to prove that the boundary is Lipschitz. To this end, note that in general we have the inclusion 
\begin{equation*}
\partial D_R\subset(\partial D\cap S_R)\cup(\partial S_R\cap D)\cup(\partial S_R\cup\partial D).
\end{equation*}
Since $S_R$ and $D$ are both open sets, for every $x\in (\partial D\cap S_R)\cup(\partial S_R\cap D)$ one can use locally the boundary representation of $D$ and $S_R$, respectively, to show the Lipschitz property. It remains to treat $x\in \partial D\cap\partial S_R$. Due to the local construction of $S_R$ above, the only non-trivial case is given by possible pentagon parts of $\partial S_R$. In the local frame there are two possibilities: the first one is $x=0$. However, the geometric construction of the pentagon yields that for $\eta>0$ small enough
\begin{equation*}
S_R\cap B_{\eta}(0)=\{(x,y)\in B_{\eta}(0):\,y>ax,\,{\rm sgn}(x)=-\sigma\}.
\end{equation*}
Due to the definition of $\sigma$ this set has no intersection with $D$, so that $0\notin\partial D_R$. The second possibility is $x=(\sigma\e,g(\sigma\e))$. There, one can check that locally the boundary of $\partial D_R$ can be parameterized in the local frame by the Lipschitz function
\begin{equation*}
\tilde{g}(x)=\begin{cases}
g(x) &\mbox{if ${\rm sgn}(x-\sigma\e)=\sigma$,}
\\
s(x) &\mbox{if ${\rm sgn}(x-\sigma\e)=-\sigma$.}
\end{cases}
\end{equation*}
This finishes the proof.
\end{proof}
}

\section*{Acknowledgements}
We warmly thank Gilles Francfort for discussions on ellipticity conditions, and acknowledge 
the financial support from the European Research Council under
the European Community's Seventh Framework Programme (FP7/2014-2019 Grant Agreement
QUANTHOM 335410).

\section*{Conflict of interest}
\noindent The authors declare that there is no conflict of interest.

\bibliographystyle{plain}
\bibliography{strongellipticity}

\end{document}